\documentclass[oneside,reqno]{amsart}
\usepackage[pdftex]{graphicx}
\usepackage{color}
\usepackage[T1]{fontenc}
\usepackage{lmodern}
\usepackage[english]{babel}
\usepackage[utf8]{inputenc}

\usepackage{upgreek}
\usepackage{amsfonts}
\usepackage{extpfeil}
\usepackage{verbatim}
\usepackage{etex}
\usepackage{microtype}

\usepackage{amsmath}
\usepackage[all, cmtip]{xy}
\usepackage{tikz-cd}
\usepackage{amssymb}
\usepackage{amsthm}
\usepackage{extpfeil}
\usepackage{bbm}
\usepackage{paralist}
\usepackage[T1]{fontenc}
\usepackage[colorlinks,citecolor=magenta,pagebackref=false,urlcolor=magenta,
pdftex]{hyperref}
\usepackage[nodisplayskipstretch]{setspace}
\usepackage{amscd}

\usepackage{nicefrac}
\usepackage{microtype}

\usepackage{mathrsfs}
\usepackage[font=small,labelfont=bf]{caption}
\usepackage[labelformat=simple]{subcaption}

\usepackage{mathabx,epsfig}
\def\acts{\mathrel{\reflectbox{$\righttoleftarrow$}}}

\makeatletter
\newtheorem*{rep@theorem}{\rep@title}
\newcommand{\newreptheorem}[2]{%
\newenvironment{rep#1}[1]{%
 \def\rep@title{#2~\ref{##1}}%
 \begin{rep@theorem}}%
 {\end{rep@theorem}}}

\makeatother

\theoremstyle{plain}
\newtheorem*{thm*}{Theorem}
\newtheorem{thm}{Theorem}[section]

\newreptheorem{mthm}{Main Theorem}
\newreptheorem{mcor}{Corollary}
\newreptheorem{mlem}{Lemma}
\newreptheorem{thm}{Theorem}

\newtheorem{cor}[thm]{Corollary}
\newtheorem{lem}[thm]{Lemma}

\newtheorem*{lem*}{Lemma}
\newtheorem{prp}[thm]{Proposition}

\theoremstyle{definition}

\newtheorem{cnvt}[thm]{\bf Convention}

\newtheorem{rem}[thm]{Remark}

\newcommand*\longhookrightarrow{\ensuremath{
\lhook\joinrel\relbar\joinrel\rightarrow}}

\newcommand*\longisorightarrow{\ensuremath{
\lhook\joinrel\relbar\joinrel\twoheadrightarrow}}

\newcommand*\longtwoheadrightarrow{\ensuremath{\relbar\joinrel\twoheadrightarrow
}}
\providecommand*{\xtwoheadrightarrow}[2][]{%
  \ext@arrow 0579\twoheadrightarrowfill@{#1}{#2}%
}

\newcommand{\supp}{\mathrm{supp}\, }

\newcommand{\wh}{\widehat}
\newcommand{\wt}{\widetilde}

 \newcommand{\R}{\mathbb{R}}

\newcommand{\vanish}[1]{}

\newcommand{\Sl}{\mr{S}}

\def\ker{\mathrm{ker}}

\def\wh{\widehat}
\def\wt{\widetilde}

\def\langle{\left<}
\def\rangle{\right>}

\def\({\left(}
\def\){\right)}
\def\no={\,{\,|\!\!\!\!\!=\,\,}}
\def\wt{\widetilde}

\def\wh{\widehat}
\def\no={\,{\,|\!\!\!\!\!=\,\,}}

\newcommand{\xqedhere}[2]{%
  \rlap{\hbox to#1{\hfil\llap{\ensuremath{#2}}}}}

\newcommand\Defn[1]{\textbf{#1}}
\newcommand{\cm}[1]{}
\newcommand\mc[1]{\mathcal{#1}}

\newcommand\mbf[1]{\mathbf{#1}}

\newcommand\mr[1]{\mathrm{#1}}

\newcommand{\K}{\Delta}

\newcommand{\bigslant}[2]{{\raisebox{.3em}{$#1$} \Big/ \raisebox{-.3em}{$#2$}}}
\renewcommand\emptyset{\varnothing}
\newcommand\x{\mathbf{x}}

\newcommand\Cl{\mr{Cl}}

\DeclareMathOperator{\Lk}{lk}

\DeclareMathOperator{\St}{st}

\title[Toric chordality]{Toric chordality}
\author{Karim Adiprasito}
\address{Einstein Institute of Mathematics, Hebrew University of Jerusalem, Jerusalem, Israel}

\email{adiprasito@math.huji.ac.il}
\date{\today}

\date{\today}

\begin{document}

\maketitle

\begin{center}
\end{center}
\begin{abstract}
We study the geometric change of Chow cohomology classes in projective toric varieties under the Weil-McMullen dual of the intersection product with a Lefschetz element. Based on this, we introduce toric chordality, a generalization of graph chordality to higher skeleta of simplicial complexes with a coordinatization over characteristic 0, leading us to a far-reaching generalization of Kalai's work on applications of rigidity of frameworks to polytope theory. In contrast to ``homological'' chordality, the notion that is usually studied as a higher-dimensional analogue of graph chordality, we will show that toric chordality has several advantageous properties and applications. 
\begin{compactitem}[$\circ$]
\item Most strikingly, we will see that toric chordality allows us to introduce a higher version of Dirac's propagation principle. 
\item Aside from the propagation theorem, we also study the interplay with the geometric properties of the simplicial chain complex of the underlying simplicial complex, culminating in a quantified version of the Stanley--Murai--Nevo generalized lower bound theorem. 
\item Finally, we apply our technique to give a simple proof of the generalized lower bound theorem in polytope theory and
\item prove the balanced generalized lower bound conjecture of Klee and Novik.
\end{compactitem}
\end{abstract}

\section*{Introduction}
A notion at the very core of graph theory, chordality is a statement about the geometry and complexity of cycles in a graph, stating in essence that a cycle is decomposable in the most economic way imaginable. The relation of graph chordality to commutative algebra in particular has motivated many to attempt a generalization of graph chordality to higher dimensions, often using a homological or combinatorial perspective (we surveyed and summarized this perspective in \cite{ANSI}). However, homological notions, as we shall argue here, are not adequate to capture graph chordality in simplicial complexes, especially not in relation to commutative algebra, and the apparent failure of several important results of graph chordality for homological chordality is a flaw of the approach rather than the problem.

As we will see, higher chordality has a better life if we add more information, and let it live in the toric variety (for simplicial polytopes at least), and is a natural symptom of the hard Lefschetz theorem, reflecting geometric changes (as measured by the size of the support) brought about by multiplication with a Lefschetz element. Philosophically, we argue that to describe the geometric behavior of simplicial chain complexes, it is better if the boundary map is close to injective. Toric geometry allows us to control the boundary map by removing redundant chains in its kernel by factoring out the torus action. 

To understand this for general simplicial complexes, we use stress spaces \cite{Lee}, a construction going back to equilibrium problems in mechanics and elastic introduced by Varignon \cite{Varignon}. It turned out much later that they form a useful model for the intersection theory of a toric variety, in particular the fact that it allows us to describe the support of a Chow cohomology class straightforwardly and combinatorially \cite{Weil, FS}. The key observation inspiring us was made by Kalai \cite{KalaiRig} (compare also \cite[Section 2.4.10]{Gromov}), using a synthetic approach to relate chordality of the graph of a simplicial polytope to minimal rigidity. 

This beautiful observation deserves a full understanding, especially so because it gives a useful picture of simplicial polytopes with vanishing of primitive second Betti numbers (with respect to the action of a Lefschetz element in the cohomology ring of the associated toric variety). Two natural questions that arise in this context are to
\begin{compactenum}[\rm (A)]
\item extend the observation to all primitive Betti numbers and, far more challenging,
\item quantify it beyond the extreme case of vanishing primitive Betti numbers.
\end{compactenum}
The first problem has been addressed in the celebrated generalized lower bound theorem of Murai and Nevo \cite{MN}, though without clarifying the beautiful relation to the geometry of homology cycles Kalai exhibited. Additionally, their proof relies rather heavily on earlier work of Green which seems hard, if not impossible, to quantify satisfyingly due to the reliance on generic initial ideals that destroy much of the combinatorial properties of the simplicial complex.

The first goal is therefore to understand, reprove and generalize Kalai's Theorem in a proper context. A secondary objective is to provide a direct and, most importantly, quantifiable proof of the generalized lower bound theorem of Murai and Nevo \cite{MN}, thereby solving both problems at once, using the newly introduced method of toric chordality. We will also apply our technique to the balanced generalized lower bound conjecture of Klee and Novik \cite{KN}.

More conceptually, this investigation is in line with combinatorial-geometric approaches to intersection theory that has seen remarkable successes disconnecting it from the algebraic varieties it classically lives on (compare \cite{MR2076929, McMullenInvent}). Toric chordality is way to study models of intersection theory in a rather precise geometric way without going to the toric variety. In particular, the present research also has some application to conjectures detailing the interplay between toric algebraic geometry and approximation theory \cite{ANSII}.

\noindent\textbf{Acknowledgements} We thank Satoshi Murai and Eran Nevo for extraordinarily useful comments and inspiring conversations concerning their work on the Generalized Lower Bound Theorem, and Robert MacPherson for pointing out some historical facts. We also thank the referee for a thorough read of this paper.

\smallskip

\section{Reminder: Homological Chordality}\label{sec:ch_hml}
Before we turn to defining chordality within models of intersection theory, it is useful to recall the homological approach to higher chordality \cite{ANSI}. While it is an immediate and naive generalization of graph chordality to higher dimensions, we will find it rather unsatisfactory.

Nevertheless, it is useful to keep in mind, as it does not depend on a geometric realization of the simplicial complex, and it is easily understandable in terms of simplicial homology.

Let $\Delta$ be a simplicial complex.
We will consider the \Defn{collection of $k$-faces} $\K^{(k)}$, and the simplicial complex of faces of dimension $\le k$, the \Defn{$k$-skeleton} $\K^{(\le k)}$ of $\K$. A \Defn{$k$-clique} is a simplicial complex of dimension $k$ that contains all possible faces of dimension $\le k$ on its vertex set, and we say that a simplicial complex has a \Defn{complete $k$-skeleton} if its $k$-skeleton is a $k$-clique. With this, one can associate to any simplicial complex its \Defn{$k$-clique complex}, defined as \[\Cl_k \K\, :=\, \{\sigma\subset \K^{(0)}: \sigma^{(\le k-1)} \subseteq \K \}.\]
The \Defn{support} of a simplicial $k$-chain (cf.\ \cite{Hatcher} for basics of simplicial homology) is the simplicial complex generated by the $k$-faces on which the chain is non-zero.

A graph $G$ is \Defn{chordal} if every simple cycle of length $\ge 4$ has a chord, i.e., an edge connecting non-adjacent vertices of the cycle. 

The homological statement this translates to is explaining that  a graph is chordal if and only if every cycle $z$ can be written as a sum of $1$-cycles of length $3$ that support no vertices that are not already vertices of $z$. 

Said in yet another way, graph chordality is equivalent to the condition that for any $1$-cycle $z$ in $\Delta=\Cl_2 G$, there exists a $2$-chain $c$ in $\Delta$ with $\partial c=z$ {and} $c^{(0)}=z^{(0)}$.

This leads us to basic notions of homological chordality. For simplicity, we restrict to homology with real coefficients here. 

We say that a $(k+1)$-chain $c\in C_{k+1}(\K)$ is a \Defn{resolution} of a $k$-cycle $z\in Z_{k}(\K)$ if $c^{(0)}= z^{(0)}$ and $\partial c=z$.  We say that a simplicial complex $\Delta$ in which every $k$-cycle admits a resolution is \Defn{(resolution) $k$-chordal}. 

The \Defn{$k$-Leray property} encodes the property that a simplicial complex is $i$-chordal for all $i\ge k$. A fundamental property of graph chordality, which goes back at least as far as to Dirac, is that graph chordality implies the Leray property immediately once the trivial obstruction vanishes. We call this the \Defn{propagation property} of graph chordality.

Recall that a \Defn{nonface} of a simplicial complex $\K$ is, naturally, a simplex on groundset $\K^{(0)}$ that is not a face of
$\K$. 
A \Defn{minimal nonface}, or \Defn{missing face}, of $\K$ is an inclusion minimal nonface of $\K$. Equivalently, a simplex $\sigma$ is a
missing face of $\K$ iff $\partial \sigma\subset \K$, but $\sigma
\not\in \K$. 

\begin{thm}[cf.\ \cite{Dirac, LekkerBoland}]
If $\Delta$ is a simplicial complex such that
\begin{compactenum}[\rm (A)]
\item $\Delta^{(1)}$ is a chordal graph. 
\item $\Delta$ has no missing faces of dimension $>1$.
\end{compactenum}
Then $\Delta$ is $1$-Leray.
\end{thm}
One natural test of strength for a generalization of chordality is whether it satisfies a propagation principle, and homological chordality falls short of this aspect in dimension $>1$.

\begin{thm}[Weak propagation property of homological chordality, cf.\ \cite{ANSI}]\label{thm:prp_exterior}
Let $\Delta$ denote any (abstract) simplicial complex without missing faces of dimension $> k$. The following are equivalent:
\begin{compactenum}[\rm (A)]
\item $\Delta$ is resolution $i$-chordal for $i\in [k,2k-1]$.
\item $\Delta$ is $k$-Leray, i.e., it is resolution $i$-chordal for $i\ge k$.
\end{compactenum}
However, for every $k\ge 2$, there is a simplicial complex $\mathfrak{J}_k$ that is
\begin{compactenum}[\rm (A)]
\item resolution $i$-chordal for $i\in [k,2k-2]$,
\item has no missing face of dimension $>k$, but
\item is not resolution $(2k-1)$-chordal.
\end{compactenum}
\end{thm}

We take this as a motivation to come up with a better notion of higher chordality that can better control the geometry of cycles.

\section{Stress groups, Stanley--Reisner theory and two highlights of combinatorial commutative algebra}

There are several important models for the intersection theory of a complete toric variety \cite{Danilov, Fulton}, but two shall be of special importance here: the intersection ring generated by divisors and Chow cohomology. Their combinatorial cousins are Stanley-Reisner rings and Minkowski weights (or McMullen--Lee stress spaces), respectively. Because toric chordality will be natural for geometric simplicial complexes beyond polytopes we will be working with these models rather than within toric geometry.

If $\Delta$ is an abstract simplicial complex on groundset $[n]:=\{1,\cdots,n\}$, let $I_\Delta:=\langle \x^{\mbf{a}}:\ \supp(\mbf{a})\notin\Delta\rangle$ denote the nonface ideal in $\R[\x]$ (cf.\ \cite{Stanley96}), where $\R[\x]=\R[x_1,\cdots,x_n]$. Let $\R[\Delta]:=\R[\x]/I_\Delta$ denote the Stanley--Reisner ring of $\Delta$. A collection of linear forms $\Theta=(\theta_1,\cdots,\theta_\ell)$ in the polynomial ring $\R[\x]$ is a \Defn{partial linear system of parameters} if \[\dim \R[\Delta]/\Theta \R[\Delta]=\dim \R[\Delta]-\ell\] for $\dim$ the Krull dimension. If $\ell=\dim \R[\Delta] = \dim \Delta +1$, then $\Theta$ is simply \Defn{linear system of parameters}.

\newcommand{\MW}{\mr{M}}

This model was used by Stanley to relate combinatorial problems for simplicial polytopes to the hard Lefschetz theorem for quasismooth projective toric varieties (see below). Rather than working with the intersection ring directly, it shall be more useful to adopt a dual perspective that in the toric setting goes back to Weil \cite{Weil} where the support of a class is more readily accessible:
Dual to the Stanley-Reisner ring in characteristic $0$ we consider the stress spaces of $\Delta$ (cf.\ \cite{Lee, TW}): 

The dual action of $\R[\x]$ acting on itself by multiplication is the action of the polynomial ring on itself by partial differentials (where naturally every variable in a polynomial $p=p(\x)$ is replaced with a corresponding partial differential $p^\vee:=p(\nicefrac{\mr{d}}{\mr{d}\x})$). 
In details, we define the \Defn{stress space} of $\Delta$ as 
\[\widetilde{\Sl}(\Delta)\ :=\ \ker\left[I_\Delta^\vee: \R[\x]\rightarrow \R[\x]\right]\]
and by minding also the linear system of parameters
\begin{equation} \label{eq:xx}
\widetilde{\Sl}(\Delta;\Theta^\vee)\ :=\ \ker\left[\Theta^\vee: \widetilde{\Sl}(\Delta)\longrightarrow \widetilde{\Sl}(\Delta) \right].
\end{equation}
The elements of this space are called stresses; the spaces themselves are very familiar in toric geometry and coincide simply with the top homology groups of the Ishida complex \cite{MR951199, MR1117638}.

Note that $\widetilde{\Sl}(\Delta;\Theta^\vee)$ and the Stanley--Reisner ring $\R[\Delta]/\Theta \R[\Delta]$ are isomorphic as graded vector spaces. Note further that the dual $c^\vee=c(\nicefrac{\mr{d}}{\mr{d} \x})$ of a linear form $c=c(\x)$ acts on the linear stress space, and therefore defines an action \[\R[\Delta]\times \widetilde{\Sl}(\Delta;\Theta^\vee)\rightarrow \widetilde{\Sl}(\Delta;\Theta^\vee).\]

We refine these definitions in several ways:

A \Defn{relative simplicial complex} $\Psi=(\Delta, \Gamma)$ is a pair of simplicial complexes $\Delta, \Gamma$ with $\Gamma \subset \Delta$.
If $\Psi=(\Delta,\Gamma)$ is a relative simplicial complex, it is not hard to modify this definition to obtain the correct picture: we simply define the \Defn{relative stress spaces} by
\[\widetilde{\Sl}(\Psi)=\bigslant{\widetilde{\Sl}[\Delta]}{\widetilde{\Sl}[\Gamma]}\]
(where $\widetilde{\Sl}[\Delta]/\widetilde{\Sl}[\Gamma]$ denotes a quotient of vector spaces)
and
\[\widetilde{\Sl}(\Psi;\Theta^\vee)\ :=\ \ker\left[\Theta^\vee:\widetilde{\Sl}(\Psi)\longrightarrow \widetilde{\Sl}(\Psi)\right]\]
and analogously for the relative Stanley--Reisner module of $\Psi$.

Observe furthermore that $\Theta$ induces a map $\Delta^{(0)}\rightarrow \R^\ell$ by associating to the vertices of $\Delta$ the coordinates $V_\Delta=({v_1},\cdots, {v_n}) \in \R^{\ell\times n}$, where $V_\Delta\x = \Theta$. Hence, as is standard to do when considering stress spaces, we identify a pair $(\Delta; \Theta^\vee)$ with a \Defn{geometric simplicial complex}, i.e.\ a simplicial complex with a map of the vertices to $\R^\ell$. The differentials given by $V_\Delta$ are therefore $V_\Delta\nabla$, where $\nabla$ is the gradient.

Conversely, the canonical stress spaces and reduced Stanley--Reisner rings, respectively, of a geometric simplicial complex are those given by the linear system of parameters given by the geometric realization, so that we usually leave out the linear system of parameters when denoting the stress space of a geometric simplicial complex.

A geometric simplicial complex in $\R^d$ is \Defn{proper} if the image of every $k$-face, $k<d$, linearly spans a subspace of dimension $k+1$.  A sequence of linear forms is a (partial) linear system of parameters if the associated coordinatization is proper.
\begin{cnvt}\label{cnvt}
The stress space of a \emph{geometric simplicial complex} is considered with respect to its natural system of parameters induced by the coordinates. We will make this clear by simply writing $\Sl(\Delta)$ for $\widetilde{\Sl}(\Delta;\Theta^\vee)$, where $\Theta$ is the linear system induced by the vertex coordinates of $\Delta$.
\end{cnvt}
 
Finally, we say a sequence of linear forms $\Theta$ of size $\ell$ is \Defn{regular} (up to degree $k$) if, for every truncation $\Theta_{j-1}=(\theta_1,\cdots,\theta_{j-1})$, $j\le \ell$, we have a surjection
\[\widetilde{\Sl}_{i+1}(\Delta;\Theta_{j-1}^\vee)\ \xtwoheadrightarrow{\  \theta_j^\vee\ }\  \widetilde{\Sl}_{i}(\Delta;\Theta_{j-1}^\vee).\]
for every $i$ (at most $k$). The \Defn{depth} of a simplicial complex is the length of the longest regular sequence that its stress space admits; this is essentially independent of the sequence of linear forms chosen as long as it is a (partial) linear system of parameters. \Defn{Cohen--Macaulay} singles out the abstract simplicial complexes for which depth equals Krull dimension. 
The following results are central to combinatorial commutative algebra, and relate topological and geometric features of the simplicial complex to the algebraic features of the ring. Recall that the  \Defn{star} and \Defn{link} of a face $\sigma$ in $\K$ are
the subcomplexes \[\St_\sigma \K\ :=\ \{\tau:\exists \tau'\supset \tau,\ \sigma\subset
\tau'\in \K\}\ \
\text{and}\ \ \Lk_\sigma \K\ :=\ \{\tau\setminus \sigma: \sigma\subset
\tau\in \K\}.\]
\begin{compactenum}[\rm (A)]
\item \emph{The Hochster--Reisner--Hibi theorem, cf.\ \cite{Stanley96}.}
A simplicial complex $\Delta$ of dimension $\ge d-1$ is of depth $\ge d$ if and only if its $(d-1)$-skeleton $\Delta^{(\le d-1)}$ is Cohen--Macaulay if and only if for every face $\sigma$ of $\Delta$, the reduced homology (with real coefficients) of $\Lk_{\sigma} \Delta$ vanishes below dimension $d-\dim \sigma -2$.
\item \emph{The hard Lefschetz theorem, cf.\ \cite{McMullenInvent, Lee}.} If $P$ is a simplicial $d$-polytope, $\Delta :=\partial P$ and $\updelta=\sum \frac{\mr{d}}{\mr{d} x_i}$ denotes the canonical differential on $\R[\x]$ (corresponding to the anticanonical divisor) associated to the homogenizing embedding $P\hookrightarrow \R^d\times \{1\} \subset \R^{d+1}$, then
\[\updelta^{d-2k}:\Sl_{d-k}(\Delta)\ \longisorightarrow\ \Sl_{k}(\Delta)\]
is an isomorphism for every $k\le \frac{d}{2}$.
\end{compactenum}

We shall use $\updelta_W$ to denote the canonical differential on vertex set $W$, i.e.\ $\updelta_W:=\sum_{w\in W} \frac{\mr{d}}{\mr{d}x_w}$.

For the results that follow, we always think of every simplicial complex as coming with an explicit coordinatization in a vector space over~$\R$. Recall that for us, it is in general not required that this realization be an embedding of the simplicial complex; rather, we usually require some properness with respect to the coordinates of vertices.

We close with a useful notion related to simplicial stresses: The space of squarefree coefficients of $\Sl(\Psi)$ is also called the space of \Defn{Minkowski weights} \cite{FS, KP}, denoted by~$\MW(\Psi)$.
\begin{prp}[cf.\ \cite{Lee}]
Consider a relative proper simplicial $d$-complex $\Psi$ in $\R^d$. Then
\[\upvarrho:\Sl(\Psi)\ \longrightarrow\ \MW(\Psi),\]
the map restricting a stress to its squarefree terms, is injective (and therefore an isomorphism). Moreover, $(\MW(\Psi))_k$ is the space of weights on $(k-1)$-faces that satisfy the \Defn{Minkowski balancing condition}: for every $(k-2)$-face $\tau$ of $\Psi$,
\[\sum_{\substack{\sigma\in\Psi^{(k-1)}\\ \sigma\supset\tau}} c(\sigma)v_{\sigma\setminus\tau}=0 \mod \mr{span}\,( \tau),\]
i.e.\ the sum on the left-hand side lies in the linear span of $\tau$.
\end{prp}
Here, we use $\Sl(\Psi)$ to denote the stress space of the geometric simplicial complex $\Psi$ as convened upon in Convention~\ref{cnvt}.

In this situation, we therefore extend the action of the Stanley--Reisner ring on the stress spaces to Minkowski weights, compare also \cite{Francois} for a direct definition of this action.

\section{The quantitative lower bound theorem (for stress spaces)} To define toric chordality, it is useful to consider first a situation where it stands out as an extreme case. We use $f_+:=\max \{f,0\}$ to denote the nonnegative part of a real function. 

Combinatorially, we think of a stress in a proper simplicial complex as the underlying \Defn{support}, i.e.\ the simplicial complex generated by faces whose coefficients are non-zero, and therefore speak of faces of a stress in a simplicial complex.

\begin{thm}[From a Chow cohomology class to its image]\label{thm:stress_quant}
We consider a simplicial polytope $P$ of dimension $d$, and its boundary complex $\Delta$. Let $g_i:=\dim \Sl_i(\Delta)-\dim \Sl_{i-1}(\Delta)$, and let $k\ge 0$ be any integer.
\begin{compactenum}[\rm (A)]
\item There exists a set $\mc{E}(\Delta)\subset \Delta^{(0)}$ of at most \[((k+1)g_{k+1}(\Delta)+ (d+1-k)g_k)_+(\Delta)\] vertices such that, for every $\gamma \in \Sl_{k+1}(\Delta)$, we have
\[\gamma^{(0)}\setminus \mc{E}(\Delta)\ \subset\ (\updelta\gamma)^{(0)}\ \subset\ \gamma^{(0)}.\]
\item The cokernel of $\updelta: \Sl_{k+1}(\Delta)\longrightarrow \Sl_{k}(\Delta)$ is of dimension $(-g_{k+1}(\Delta))_+$.
\end{compactenum}
\end{thm}

The second part is a simple conclusion of the hard Lefschetz theorem; for the first fact, we observe an auxiliary lemma. 
The \Defn{boundary of a star} is 
\[\partial \St_\sigma \K\ :=\ \{\tau:\sigma \nsubseteq \tau, \exists \tau'\supset \tau,\ \sigma\subset
\tau'\in \K\}\] and the \Defn{open star} of $\sigma$ in $\Delta$ is the relative complex $\St_\sigma^\circ (\Delta):=(\St_\sigma \Delta, \partial \St_\sigma \Delta).$

\begin{lem}\label{lem:quant}
In the situation of Theorem~\ref{thm:stress_quant}, the map
\begin{equation}\label{eq:st}
\Sl_{k+1}(\St_v^\circ \Delta)\ \xrightarrow{\ \ \updelta\ }\ \Sl_{k}(\St_v^\circ \Delta), \ \ v\in \Delta^{(0)},
\end{equation}
is an injection for all but at most $((k+ 1)g_{k+1}+ (d+ 1-k)g_k)(\Delta)$ vertices of $\Delta$.
\end{lem}

Let us recall two simple facts:

\begin{compactenum}[\rm (A)]
\item \emph{McMullen's integral formula for the $g$-vector, cf.\ \cite[Prop.4.10]{Swartz:g-elements}).}
For any proper geometric simplicial $(d-1)$-sphere~$\Delta$ in $\R^d$, and any integer $k$, we have
\[\sum_{v\in \Delta^{(0)}} g_k(\Lk_v \Delta) = ((k+ 1)g_{k+1}+ (d+ 1-k)g_k)(\Delta)\]
\item \emph{Cone Lemma I (\cite[Cor.1.5]{TWW} \& \cite[Thm.7]{Lee}).}
For any vertex $v\in \Delta$, $\Delta$, a geometric simplicial complex in $\mathbb{R}^d$, any integer $k$, and $\wh{\Lk}_v \Delta$ the orthogonal projection of $\Lk_v \Delta$ to the orthogonal complement of the vector $v$, we have an isomorphism
\[\Sl_k(\wh{\Lk}_v \Delta)\ \cong\ \Sl_{k}(\St_v \Delta).\]
\end{compactenum}

McMullen and Swartz state the integral formula only for the $h$-vector (defined in this context by $h_i(\Delta)=\dim \Sl_i(\Delta)$), but the form for the $g$-vector given here follows directly from the underlying identity for Hilbert functions implicit in Swartz' proof. For an appropriate combinatorial definition of $g_k$, this fact in particular extends to all simplicial complexes.

We need also a second version of the cone lemma that corresponds to a pullback in Chow rings.

\begin{lem}[Cone Lemma II]
In the situation of the first cone lemma and $\Delta$ in $\R^n$ we have a natural isomorphism
\[\updelta_v:\Sl_{k+1}(\St_v^\circ \Delta)\ \longisorightarrow\ \Sl_{k}(\St_v \Delta).\]
\end{lem}

\begin{proof}
The map \[\updelta_v:\widetilde{\Sl}_{k+1}(\St_v^\circ \Delta)\longrightarrow \widetilde{\Sl}_{k}(\St_v \Delta)\] is clearly an isomorphism. Let $(\vartheta_i)$ denote an orthogonal basis of $(\mathbb{R}^d)^\vee$. Then the preimage of 
$\updelta_v^{-1}{\Sl}_{k}(\St_v \Delta) \subset \widetilde{\Sl}_{k+1}(\St_v^\circ \Delta)$ vanishes under the elements of $(\vartheta_i)$ by injectivity of $\updelta_v$, so that \[\updelta_v:{\Sl}_{k+1}(\St_v^\circ \Delta) \longrightarrow {\Sl}_{k}(\St_v \Delta)\]
is surjective; injectivity follows as the map is obtained as restriction of an injective map.
\end{proof}

\begin{proof}[\textbf{Proof of Lemma~\ref{lem:quant}}]
It suffices to show that $g_k(\Lk_v \Delta)\le 0$ for all but at most $(k+ 1)g_{k+1}+ (d+ 1-k)g_k$ vertices. But $g_k(\Lk_v \Delta)\ge 0$ for all $v$ by the hard Lefschetz theorem, so the claim follows by McMullen's integral formula.
\end{proof}

\begin{proof}[\textbf{Proof of Theorem~\ref{thm:stress_quant}}]
Claim (B) is clear. It remains to discuss claim (A). 

By Lemma~\ref{lem:quant}, the map~\eqref{eq:st} is an isomorphism except for a small set $\mc{E}$ of vertices $v$.
In particular, if  $\gamma \in \Sl_{k+1}(\St_v \Delta)$ is a stress, and we consider the restriction $\gamma_{v} \in \Sl_{k+1}(\St_v^\circ \Delta)$ of $\gamma$ to the open star of $v$ in $\Delta$, then $\updelta \gamma_v \neq 0$ unless $v \in \mc{E}$. 
\end{proof}

\section{Toric chordality and Dirac-Green propagation}
The extreme cases of Theorem~\ref{thm:stress_quant} lead us to consider toric chordality: Consider a geometric simplicial complex $\Delta$, and a linear differential $\upomega$. We say that $(\Delta,\upomega)$ is \Defn{toric $k$-chordal} for some integer $k$ if $\upomega$ induces a surjection
\begin{equation}\label{eq:tiso}
\Sl_{k+1}(\Delta)\ \xtwoheadrightarrow{\ \upomega\ }\ \Sl_{k}(\Delta)
\end{equation}
and an injection 
\begin{equation}\label{eq:tinj}
\Sl_{k}(\Delta)\ \xhookrightarrow{\ \  \, \upomega\ \  \, }\ \Sl_{k-1}(\Delta)
\end{equation}
We say the pair $(\Delta,\upomega)$ is \Defn{weakly toric $k$-chordal} if 
\begin{equation}\label{eq:tinj2}
\Sl_{k}(\St_v \Delta)\ \xhookrightarrow{\ \  \, \upomega\ \  \, }\ \Sl_{k-1}(\St_v \Delta)
\end{equation}
is an injection for all vertices $v$ of $\Delta$.

\begin{rem}
It is useful to remark that the three conditions for toric chordality do not come without interplay.
\begin{compactenum}[\rm (A)]
\item We have $\Sl_{k}(\St_v \Delta) \hookrightarrow \Sl_{k}(\Delta)$, so injection~\eqref{eq:tinj2} is implied by~\eqref{eq:tinj}. 
\item If $\Delta$ is a proper geometric simplicial complex in $\R^d$, then toric $d$-chordality implies resolution $(d-1)$-chordality.
Indeed, in this situation and following the action $\R[\Delta]\acts \MW(\Delta)$, we have a natural isomorphism
\begin{equation}\label{eq:iso_tay}
\Sl_{d}(\Delta)/\updelta\Sl_{d+1}(\Delta)\ \cong \ \widetilde{H}_{d-1} (\Delta;\R),
\end{equation}
compare also Ishida \cite{MR951199, MR1117638} and equivalently Tay--Whiteley \cite{TW}. Indeed, following Ishida, stress spaces are but special simplicial chains and cycles with real coefficients.
\end{compactenum}
\end{rem}
We observe as above:
\begin{lem}\label{lem:quant2}
If $(\Delta,\upomega)$ is weakly toric $k$-chordal for some $k\ge 0$, then for every $\gamma$ in $\Sl_{k+1}(\Delta)$, we have
\[(\upomega\gamma)^{(0)}\ =\ \gamma^{(0)}.\]
\end{lem}
\begin{proof}
Argue as in the proof of Theorem~\ref{thm:stress_quant}.
\end{proof}

It follows in particular that the map
$\upomega:\Sl_{k+1}(\Delta) \rightarrow \Sl_{k}(\Delta)$ is injective.

\begin{cor}\label{cor:quant2}
Weak toric chordality propagates, i.e.\ a weakly toric $k$-chordal complex, $k\ge 0$, is also weakly toric $\ell$-chordal for every $\ell \ge k$.
\end{cor}

Let us justify the notion of toric chordality by reconsidering Theorem~\ref{thm:stress_quant}:

\begin{prp}
Consider the boundary of a simplicial $d$-polytope $\Delta$. If $g_k(\Delta)=g_{k+1}(\Delta)=0$ (for some $k\ge 0$), then $\Delta$ is toric $k$-chordal.
\end{prp}

\begin{proof}
This follows at once from the Lefschetz theorem.
\end{proof}

The main motivation for toric chordality, however, is that it satisfies the propagation principle; clearly in the form of Theorem~\ref{thm:prp_exterior}, but even after it is satisfied only in a single degree.

\begin{thm}[Higher Dirac ``propagation principle'']\label{thm:toric_propagation}
Assume that, , for some $k\ge 0$, a geometric properly embedded simplicial complex $\Delta$
\begin{compactitem}[$\circ$]
\item is toric $k$-chordal w.r.t.\ a linear differential $\upomega$ and that
\item $\Delta$ has no missing faces of dimension $k+1$.
\end{compactitem}
Then $\Delta$ is toric $(k+1)$-chordal.
\end{thm}

\begin{proof}
	
Consider a $(k+1)$-stress $\gamma$ in $\Delta$, and its restriction $\gamma_v$ to $\St_v^\circ \Delta$ for some $v\in \Delta^{(0)}$.
Then $\updelta_v \gamma_v$ is, by assumption, in the image of $\upomega: \Sl_{k+1}(\Delta) \rightarrow \Sl_{k}(\Delta) \supset \Sl_{k}(\St_v \Delta)$.

We proceed to show that there exists a relative $(k+2)$-stress $\wt{\gamma}_v$ in $\St_v^\circ \Delta$ that maps to $(\upomega^{-1}\circ\updelta_v)(\gamma_v)\in\Sl_k(\Delta)$ under $\updelta_v$.
To this end we show that $(\upomega^{-1}\circ\updelta_v)(\gamma_v)\in \Sl_{k+1}( \Delta)$ is supported in $\St_v \Delta$.

Since there are no missing faces of cardinality $k+2$, a $k$-face $\sigma,\ v\notin \sigma,$ is in $\St_v \Delta$ if and only if $\partial \sigma$ is in $\St_v \Delta$. Hence, it suffices to show that for every $w\neq v\in \Delta^{(0)}$, the stress \[\gamma_{v,w}\ :=\ \updelta_w(\upomega^{-1}\circ\updelta_v)(\gamma_v) \in \Sl_k(\St_w \Delta)\] is supported in the star of~$v$.

Consider the commutative diagram
\[\begin{tikzcd}
& \Sl_{k}(\Delta) \arrow[hookrightarrow]{r}{\ \upomega\ }  &  \Sl_{k-1}(\Delta)&\\
\Sl_{k+1}(\St_v^\circ \Delta) \arrow{r}{\ \upomega\ } \arrow{ur}{\updelta_v} &  \Sl_{k}(\St_v^\circ \Delta) \arrow{ur}{\updelta_v} & &
\end{tikzcd}
\]
The top horizontal map $\upomega$ is an injection, and the bottom map is a surjection 
when restricted to the image of the restriction map \[\Sl_{\ast}(\Delta)\ \longrightarrow\ \Sl_{\ast}(\St_v^\circ \Delta)\]
Hence, the inverse of $\upomega\gamma_{v,w} \in \Sl(\St_{\{v,w\}} \Delta)$
under 
\[\Sl_{k+1}(\St_v^\circ \Delta)\ \xrightarrow{\ \upomega\ }\ \Sl_{k}(\St_v^\circ \Delta)\ \xrightarrow{\ \updelta_v\ }\ \Sl_{k-1}(\St_{v} \Delta)\]
maps to $\gamma_{v,w}$ under $\updelta_v$. Hence $\gamma_{v,w}$ is a stress supported in the star of $v$, which implies that $(\upomega^{-1}\circ\updelta_v)(\gamma_v)$ is supported in the star of $v$. Via the cone lemma, we obtain as desired a preimage for $\gamma_v$ under $\upomega$ in $\Sl_{k+2}(\St_v^\circ \Delta)$.

Finally, to obtain the surjection
\[\Sl_{k+2}(\Delta) \xtwoheadrightarrow{ \upomega } \Sl_{k+1}(\Delta).\]
let $\gamma$ denote a $(k+1)$-stress in $\Delta$, and let $\gamma_v$ denote is restriction to the relative stress in $\St_v^\circ \Delta$. By the previous argument, we construct relative stresses $\wt{\gamma}_v$ that map to $\gamma_v$ under $\upomega$.

Now, by the cone lemma and weak toric chordality, $\upomega$ is injective as a map from $\Sl_{k+2}(\St_{e}^\circ \Delta)$ to $\Sl_{k+1}(\St_{e}^\circ \Delta)$ for every edge $e=\{v,w\}$ of~$\Delta$. Hence, $\wt{\gamma}_v$ and $\wt{\gamma}_w$, for vertices $v$, $w$ of $\Delta$, coincide in $\St_{v}^\circ \Delta\cap \St_{w}^\circ \Delta=\St_{e}^\circ \Delta$. Hence, there exists a stress $\wt{\gamma} \in \Sl_{k+2}(\St_{e}^\circ \Delta)$ that restricts to 
$\wt{\gamma}_v$ on $\St_{v}^\circ \Delta$: it is given by assigning the monomial $x^\alpha$, $v \in \supp \alpha$, the coefficient of $x^\alpha$ in $\wt{\gamma}_v$. 

It follows that $(\Delta,\upomega)$ is toric $(k+1)$-chordal.
\end{proof}

\begin{cor}\label{cor:vanishStressHvecor}
Consider a proper geometric simplicial complex in $\R^d$, and assume $\Delta$ has no missing faces in dimension $> k$ for some $k\ge 0$.
If $(\Delta,\upomega)$ is toric $k$-chordal, then \[\ker[\upomega:\Sl_{i+1}(\Delta) \rightarrow \Sl_{i}(\Delta)]=0\] for all $i\ge k$.
\end{cor}

The following fact is also a simple consequence of Theorem~\ref{thm:toric_propagation}.

\begin{cor}\label{cor:propagation_SR}
Let $\Delta$ denote an abstract simplicial complex, $\upomega$ a linear differential, and $V_\Delta$ a coordinatization of $\Delta$ in a real vector space. Assume that, for some $k\ge 0$,
\begin{compactenum}[\rm (A)]
\item with respect to the coordinatization and the differential $\upomega$, $\Delta$ is toric $k$-chordal, and
\item $\Delta$ has no missing faces of dimension $> k$, and
\item that $(V_\Delta\nabla,\upomega)$ is regular up to degree $k$ on $\Sl(\Delta)$.
\end{compactenum}
Then $(V_\Delta\nabla, \upomega)$ is a regular system of parameters, and $\Delta$ is Cohen--Macaulay.
\end{cor}


\begin{proof}
This fact is very easy to see when adopting the algebraic perspective: it uses the simple fact that a graded module $M$ over $\mathbb{R}[x_1,\cdots,x_d]$ such that
\[\times y_d:\ \bigslant{M}{\langle y_1,\cdots,y_{d-1}\rangle}\ \longrightarrow\ \bigslant{M}{\langle y_1,\cdots,y_{d-1}\rangle}\]
is injective for every linear coordinate change $(x_1,\cdots,x_d) \rightarrow (y_1,\cdots,y_d)$ is free.

The proof only needs one easy observation; to simplify notation, let us just look at the case $d=2$. How would we show that the multiplication by $x_1x_2$ defines an injective map on $M$? Simply because we can exploit the linear transformation $(x_1,x_2) \rightarrow (x_1-x_2,x_1+x_2)$, and use the assumption coupled with the fact that
\[(x_1+x_2)^2\ =\ 2x_1x_2 \ \mod (x_1-x_2). \qedhere\]
\end{proof}
\section{Interlude: Dirac's cut theorem in the toric case}

Dirac's cut theorem is a second influential theorem on chordal graphs, cf.\ \cite{ANSI, MR2083448}; acting in the opposite direction of the propagation theorem, and is a cornerstone of results in chordal graph theory.
\begin{thm}[Dirac, \cite{Dirac}]
Let $G=G_1 \cup G_2$ denote a graph, and let $\overline{G}=G_1\cap G_2$ denote its induced intersection (i.e.\ their intersection is an induced subgraph of $G$). Assume that
\begin{compactenum}[(A)]
\item we have a surjections
\[\partial: Z_1(G_i,\overline{G})\ \longrightarrow\ Z_0(\overline{G}) \]
for $i=1, 2$, and 
\item that $G$ is chordal.
\end{compactenum}
Then $\overline{G}$ is connected. If, moreover, a cycle $z\in Z_0(\overline{G})$ has preimages $z_i \in Z_1(G_i,\overline{G}),\ i=1,2$ so that $\overline{G} \cap z_i^{(0)}= z^{(0)}$ for $i=1,2$, then $z$ has a resolution in $\overline{G}$ (in the sense of resolution chordality).
\end{thm}

Again, a version of this theorem also holds in the toric case:

\begin{thm}
Let $\Delta=\Delta_1 \cup \Delta_2$ denote a geometric simplicial complex in $\R^{d-1}$, and let $\overline{\Delta}=\Delta_1\cap \Delta_2$ denote their induced intersection (i.e.\ $\overline{\Delta}$ is an induced subcomplex of $\Delta$). 
Assume that, for two differentials $\uppsi, \upomega$, and some $k\ge 0$,
\begin{compactenum}[(A)]
\item we have surjections
\[\uppsi: \Sl_k(\Delta_i,\overline{\Delta})\ \longrightarrow\ \Sl_{k-1}(\overline{\Delta}) \]
 for $i=1, 2$, and
\item $(\ker [\uppsi: \Sl_\ast(\Delta) \rightarrow \Sl_\ast(\Delta)] ,\upomega)$ is toric $k$-chordal.
\end{compactenum}
Then \[\upomega:\Sl_{k}(\overline{\Delta},\upomega)\ \longrightarrow \Sl_{k-1}(\overline{\Delta},\upomega)\] surjective.

If, moreover, a stress $\gamma\in \Sl_{k-1}(\overline{\Delta})$ has preimages $\pi_i \in \Sl_k(\Delta_i,\overline{\Delta}),\ i=1,2$ so that $\overline{\Delta} \cap \pi_i^{(0)}= \gamma^{(0)} $ for $i=1,2$, then $\upomega^{-1}\gamma$ can be chosen so that $(\upomega^{-1}\gamma)^{(0)}=\gamma^{(0)}$.
\end{thm}

Here, \[\widetilde{S}:=\ker [\uppsi: \Sl_\ast(\Delta) \rightarrow \Sl_\ast(\Delta)] \] can naturally be seen as the stress space of a geometric simplicial complex in $\R^d$, justifying the second condition.

\begin{proof}
For a $(k-1)$-stress $\gamma$ in $\Sl_{k-1}(\overline{\Delta})$, consider its preimages $\pi_i\in \Sl_k(\Delta_i,\overline{\Delta})$, $i=1, 2,$ under~$\uppsi$. 

\begin{figure}[h!tb]
\begin{center}
\includegraphics[scale = 0.6]{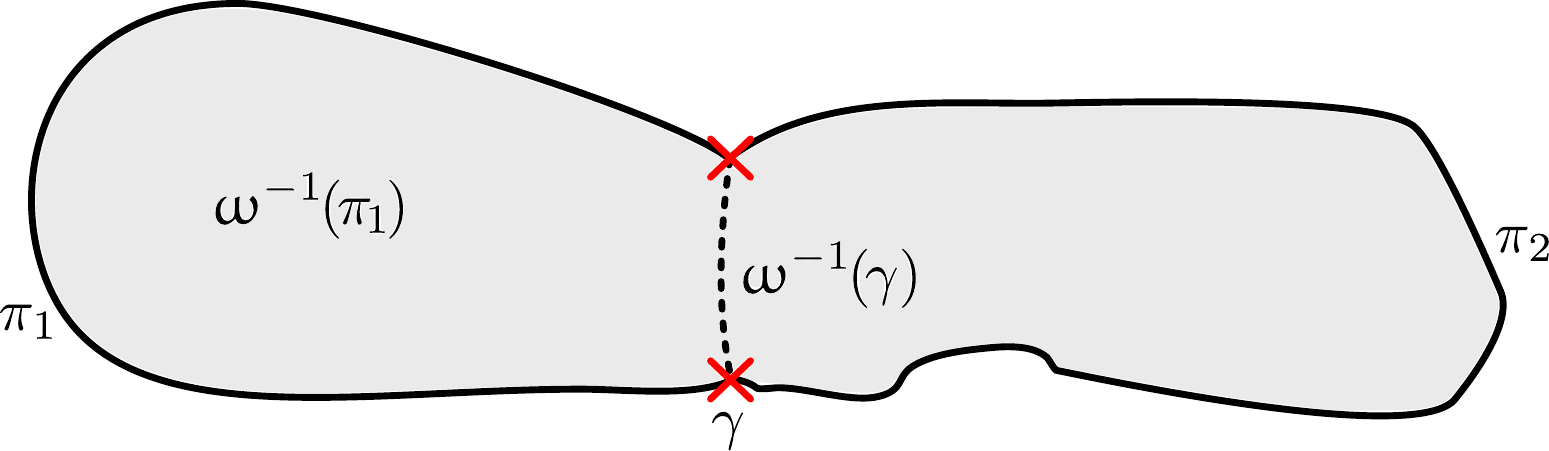}
\caption{Dirac's cut theorem in the toric case.}
\label{fig:cuts}
\end{center}
\end{figure}

Then $\pi_1-\pi_2$ is a stress
in $\widetilde{S}_k$ (as $\uppsi (\pi_1) = \uppsi (\pi_2)$), so that $\upomega^{-1}(\pi_1-\pi_2)$, restricted to $\Sl_{k+1}(\Delta_i,\overline{\Delta})$, is a preimage $\alpha_i$ of $\pi_i$ under $\upomega$. But then 
$\uppsi(\alpha_i)$ is a stress in $\Sl_k(\overline{\Delta})$, and
\[(\upomega\circ \uppsi)(\alpha_i)\ =\ \gamma,\]
as desired. The second part of the theorem follows if we consider the support of $\upomega\uppsi(\alpha_i)$.
\end{proof}

\section{Relation to resolution chordality}
So far, toric chordality is only an interesting concept, without any clear connection to homological chordality studied classically. Let us try to remedy this by translating Theorem~\ref{thm:stress_quant} to a statement about simplicial homology. The results of this section refine observations by Kalai \cite{KalaiRig} from graph cycles to general homology cycles; see also \cite{ANSII} for a generalization of his ideas for geometrically restricted homology cycles. We use toric chordality instead. 

\begin{lem}[Toric $k$-chordality and resolution chordality, I]\label{lem:torreschord}
Let $\Delta$ in $\R^d$ be any proper geometric simplicial complex. Assume that, and some $k\ge 0$,
\begin{compactenum}[\rm (A)]
\item $\widetilde{H}_{k-1}(\Lk_\sigma \Delta)=0$ for every $\sigma \in \Delta$ of cardinality $< d-k$,
\item $\dim \mr{coker}[\upomega:\Sl_{k+1}(\Delta)\ \xrightarrow{\ \upomega\ }\  \Sl_{k}(\Delta)]= \alpha$ and
\item that $(\Delta,\upomega)$ is weakly toric $k$-chordal.
\end{compactenum}
Then $\wt{\beta}_{k-1}(\Delta_{W})\ \le\ \alpha$ for every $W\subset \Delta^{(0)}$. 

In particular, if $(\Delta,\upomega)$ is toric $k$-chordal and (A) holds then $\Delta$ is resolution $(k-1)$-chordal, and has in particular no missing $k$-faces.
\end{lem}

Here $\Delta_{W}$ is the subcomplex of $\Delta$ induced by a vertex set $W\subset \Delta^{(0)}$, and $\wt{\beta}$ denotes reduced Betti numbers with real coefficients.

\begin{thm}\label{cor:torreschord}
Let $\Delta$ be the boundary of any simplicial $d$-polytope, and let $k\ge \frac{d}{2}$. Then
\[\wt{\beta}_{k-1}(\Delta_{W})\le -g_{k+1}(\Delta)\] for $W\subset \Delta^{(0)}$.
\end{thm}

 \begin{rem}
 Theorem~\ref{cor:torreschord} alone can also be proven rather elegantly using the stratification of the toric variety by torus orbits, which can be thought of as a diagram of spaces built over the intersection poset of $P$. Studying the associated resolution of the Ishida complex gives the desired fact easily (compare also \cite{ANSII}). Lemma~\ref{lem:torreschord} can be proven this way as well with just a little more work.
\end{rem}

\begin{rem}
The situation of $\Delta\cong S^{d-1}$ has an advantage, as Alexander duality and the Dehn--Sommerville relations allow us to make a statement for $k\le \frac{d}{2}$ as well.
\end{rem}

\begin{cor}
Let $\Delta$ be the boundary of any simplicial $d$-polytope, and any $k\le \frac{d}{2}$, 
\[\wt{\beta}_{k-1}(\Delta_{W})\le g_{k}(\Delta)\ \text{for}\ W\subset \Delta^{(0)}.\] 
\end{cor}

\begin{rem}
This corollary alone can also be shown using the Mayer--Vietoris resolution of the \v{C}ech complex of local stress spaces as in \cite{ANSII}.
\end{rem}

Compare also \cite{ANSII} for an application to a conjecture of Kalai.

\begin{proof}[\textbf{Proof of Lemma~\ref{lem:torreschord}}]
 
To prove the desired inequality, we provide iteratively an embedding of homology classes associated to certain critical Morse loci into the to the space of coprimitive stresses 
$\Sl_{i}(\Delta)/\upomega\Sl_{i+1}(\Delta)$.
We argue by induction on $d-k$, where we note that the case $d-k=0$ is done already: 

\textbf{Induction start:} To start with, assumption (C) gives us an injection \[\mbf{i}_{\Delta_{W}}^\Delta[d]:\wt{H}_{d-1}(\Delta_{W})\ \longhookrightarrow\ \Sl_{d}(\Delta)/\upomega\Sl_{d+1}(\Delta)\] from homology classes to coprimitive stresses, see isomorphism~\eqref{eq:iso_tay}. 

\textbf{Induction step:} The key observation for the induction is that, since \[\wt{H}_{k-1}(\Delta)\ =\ \wt{H}_{k-1}(\Delta_{\Delta^{(0)}})\ =\ 0\] following assumption (A), homology classes of $\wt{H}_{k-1}(\Delta_{W})$ correspond to relative classes in $\wt{H}_{k}(\Delta,\Delta_{W})$. 

To exploit this, we use Morse theory, following the construction of Morse homology: remove the vertices of $\Delta$ one by one, exploring the complex as a Morse function does, until we arrive at the desired vertex set $W$. We use the Morse function to build an injection
\[\mbf{i}_{\Delta_{W}}^\Delta[k]:\wt{H}_{k-1}(\Delta_{W})\ \longhookrightarrow\ \Sl_{k}(\Delta)/\upomega\Sl_{k+1}(\Delta).\]
If, in some stage of removal, the map
\[\iota_{X,v,\Delta,k}:\wt{H}_{k-1}(\Delta_{X\setminus \{v\}})\ \longrightarrow\ \wt{H}_{k-1}(\Delta_{X})\]
induced by inclusion is \emph{not} a injection, then we call $v$ a \emph{critical point}. 

Given the sequence 
\[\wt{H}_{k}(\Delta_{X},\Delta_{X\setminus \{v\}})\ \longrightarrow\ \wt{H}_{k-1}(\Delta_{X\setminus \{v\}})\ \longrightarrow\ \wt{H}_{k-1}(\Delta_{X})\]
we see that the boundaries of relative classes in $\wt{H}_{k}(\Delta_{X},\Delta_{X\setminus \{v\}})\cong\wt{H}_{k}(\St^\circ_{v} \Delta_{X})$ generate the kernel of $\iota_{X,v,\Delta,k}$. Let $\mathcal{A}_{X,v,\Delta}$ denote a subspace of $\wt{H}_{k}(\Delta_{X},\Delta_{X\setminus \{v\}})$ so that 
\[0\ \longrightarrow\ \mathcal{A}_{X,v,\Delta,k}\ \longrightarrow\ \wt{H}_{k-1}(\Delta_{X\setminus \{v\}})\ \longrightarrow\ \wt{H}_{k-1}(\Delta_{X})\]
is exact.

By the cone lemmas, we can think of $\St_v^\circ\Delta_X$ as the nonrelative simplicial complex $\Lk_v \Delta_X$ in $\R^d/\langle v\rangle$, so that we can apply the induction assumption to $\Lk_v \Delta_X$ and obtain an embedding \[\mbf{i}_{\St_v^\circ \Delta_{X}}^{\St_v^\circ\Delta}[k+1]: \wt{H}_{k}(\Delta_{X},\Delta_{X\setminus \{v\}})\ \longhookrightarrow\ \Sl_{k+1}(\St_v^\circ\Delta_X)/\upomega\Sl_{k+2}(\St_v^\circ\Delta_X).\]

We compose this with the partial differential \[\updelta_v:\Sl_{k+1}(\St_v^\circ\Delta_X)/\upomega\Sl_{k+2}(\St_v^\circ\Delta_X)\ \longrightarrow\ \Sl_{k}(\St_v \Delta_X)/\upomega\Sl_{k+1}(\St_v \Delta_X)\] and obtain a map to $\Sl_{k}(\St_v\Delta_X)/\upomega\Sl_{k+1}(\St_v\Delta_X)$ via the cone lemma, obtaining an embedding 
\[\mathcal{A}_{X,v,\Delta,k}
\ \xhookrightarrow{\ \updelta_v\circ {\mbf{i}}_{\St_v^\circ \Delta_{X}}^{\St_v^\circ\Delta}[k+1]\ }\ \Sl_{k}(\St_v\Delta_X)/\upomega\Sl_{k+1}(\St_v\Delta_X).\]
Since, by assumption, elements of the kernel of $\iota_{X,v,\Delta,k}$ generate homology classes of dimension $k-1$ in $\Delta_{X\setminus\{v\}}$, the image of this embedding lies in $\Sl_{k}(\Delta_X)/\upomega\Sl_{k+1}(\Delta_X)$, so that we can see the previous map as an embedding
\[\mathcal{A}_{X,v,\Delta,k}
\ \xhookrightarrow{\ \updelta_v\circ {\mbf{i}}_{\St_v^\circ \Delta_{X}}^{\St_v^\circ\Delta}[k+1]\ }\ \Sl_{k}(\Delta_X)/\upomega\Sl_{k+1}(\Delta_X).\]

We now proceed iteratively, and remove vertex after vertex from $\Delta^{(0)}$. The combination of the maps constructed in each step gives a map 
\[\widetilde{\mbf{i}}_{\Delta_{W}}^\Delta[d]: \bigoplus_{j=1,\cdots,n} \mathcal{A}_{X_j,v_j,\Delta,k} \ \longrightarrow\ \Sl_{k}(\Delta)/\upomega\Sl_{k+1}(\Delta)\]
where $(v_j)_{j=1,\cdots,n}$ is a sequence of vertices in $\Delta^{(0)}$ with $\Delta^{(0)}\setminus \{v_j : j=1,\cdots,n\}=W$ and $X_j:=\Delta^{(0)}\setminus \{v_j : j=1,\cdots,j-1\}$. Since every homology class in dimension $k-1$ has a unique representation as linear combination of elements in  $\mathcal{A}_{X_j,v_j,\Delta,k}$ (it was chosen to be a basis of the generated homology at the critical points of the Morse function, after all), this induces a map 
\[H_{k-1}(\Delta_W) \ \longrightarrow\ \Sl_{k}(\Delta)/\upomega\Sl_{k+1}(\Delta)\]

\begin{figure}[h!tb]
\begin{center}
\includegraphics[scale = 0.6]{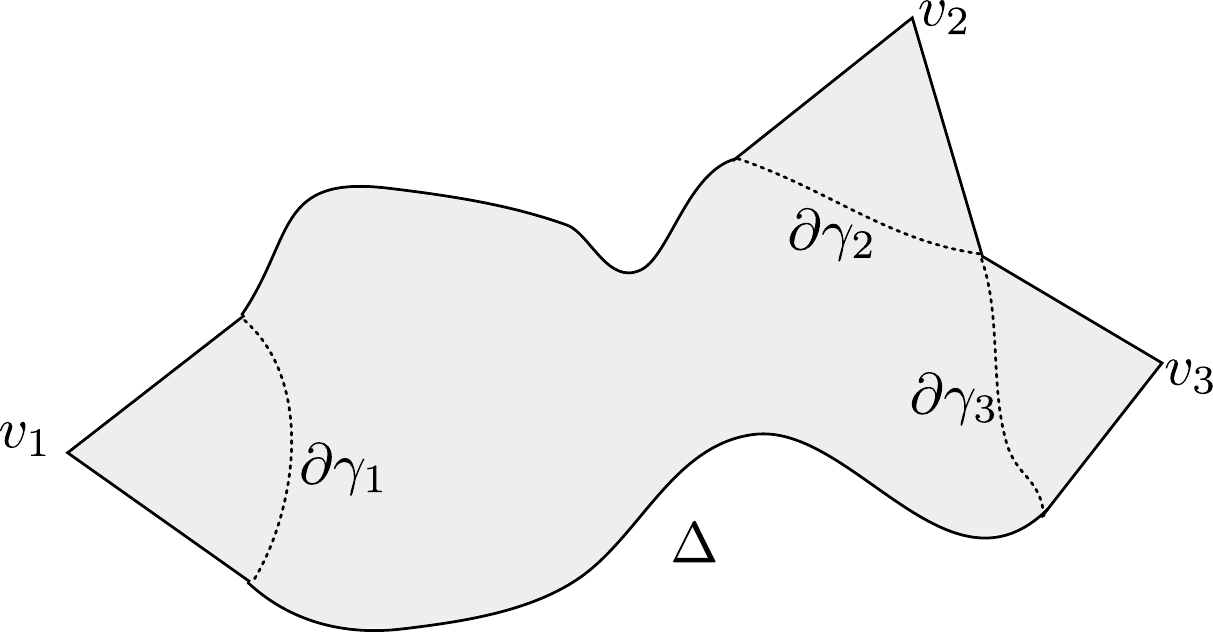}
\caption{Constructing the embedding by induction.}
\label{fig:map}
\end{center}
\end{figure}

To see that this map is an injection, assume the contrary. 
Since the map is locally (i.e.\ in every construction step when removing a vertex) an injection, there must be $\gamma_j \in \mathcal{A}_{X_j,v_j,\Delta,k}$ so that \[\sum_{j=1}^n (\updelta_v\circ {\mbf{i}}_{\St_v^\circ \Delta_{X}}^{\St_v^\circ\Delta}[k+1])(\gamma_j) \in {\Sl_{k}(\Delta)}\ =\ \overline{\gamma}\ \in\ \upomega\Sl_{k+1}(\Delta).\]

Let $j_0$ denote the smallest $j$ such that $\gamma_j\neq 0$. But then $\upomega^{-1}\overline{\gamma}\in \Sl_{k+1}(\Delta_{X_{j_0}})$ is a chain supported in $X_{j_0}$ by weak toric chordality. Hence $\partial\gamma_{j_0}$ is a boundary in $\Delta_{X_{j_0} \setminus \{v_{j_0}\}}$, in contradiction with the construction.
\end{proof}

\begin{rem}[Birth, death, and higher connectivity]
Note that the proof of Lemma~\ref{lem:torreschord} was very wasteful, only recording the birth, not the death, of homology classes.
In particular, it is immediate that the inequalities above are not sharp, especially when the primitive Betti numbers are far from minimal. It would be interesting to understand whether the inequalities, or the connection between stress spaces and ordinary homology in general, could reveal some interesting results on higher connectivity of polytope graphs and skeleta. 

In the same vein, the perspective via toric geometry could also be  interesting to study higher-dimensional expansion (cf.~\cite{lubotzky}), in particular in connection with Hodge theory of toric varieties. Indeed, we saw here that toric chordality is an effective approach to study what is essentially the case of minimal connectivity, while higher expansion focuses on the case of maximum connectivity, so that it seems only natural to think that both can be studied in the same setting.
 \end{rem}

For a dual result, we consider the setting in which chordality fails at a stable set of vertices.

\begin{lem}[Toric $k$-chordality and resolution chordality, II]\label{thm:torreschor}
Let $\Delta$ in $\R^d$ be any proper geometric simplicial complex. Assume that, for some $k\ge 0$,
\begin{compactenum}[\rm (A)]
\item $\widetilde{H}_{k-1}(\Lk_\sigma \Delta)=0$ for every $\sigma \in \Delta$ of cardinality $< d-k$,
\item $\Sl_{k}(\St_v \Delta)\ \xhookrightarrow{\ \ \, \upomega\ \ \, }\ \Sl_{k-1}(\St_v \Delta)$ is an injection for all but a set $\mc{E}$ of vertices $v$ of $\Delta$, and
\item that $\Sl_{k+1}(\Delta)\ \xtwoheadrightarrow{\ \upomega\ }\ \Sl_{k}(\Delta)$ is a surjection.
\end{compactenum}
Then $\widetilde{\beta}_{k-1}(\Delta_{W\cup \mc{E}})=0$ for every subset $W\subset\Delta^{(0)}$.
\end{lem}

\begin{proof}
The proof is analogous to the one of Lemma~\ref{lem:torreschord}, this time taking care that none of the ``bad'' vertices of $\mc{E}$ get removed in the iterative removal of vertices going from $\Delta^{(0)}$.

Let $v\notin \mc{E}$ be such a removed vertex, going from a vertex set $X$ containing $v$ and $\mc{E}$ to the vertex set $X\setminus \{v\}$. Following the proof of Lemma~\ref{lem:torreschord} above, a supposed homology $(k-1)$-class $\gamma$ generated in the attaching step corresponds injectively to a stress in ${\Sl_{k}(\St_v \Delta)}/{\upomega\Sl_{k+1}(\St_v\Delta)}$. But such a stress is in the image of $\upomega$ in $\Delta$; so that $\gamma$ is a boundary in $\Delta_{X\setminus \{v\}}$; a contradiction.
\end{proof}

We formulate a useful consequence.

\begin{thm}\label{cor:betti}
Let $\Delta$ be the boundary of any simplicial $d$-polytope, and let $g_{k+1}(\Delta)\ge 0$ for some $k\ge 0$.
Then there exists a subset $\mc{E}\subset \Delta^{(0)}$, \[\# \mc{E}\ \le\ ((k+1)g_{k+1}+(d+ 1-k)g_k)(\Delta)\] such that $\widetilde{\beta}_{k-1}(\Delta_{W\cup \mc{E}})=0$ for $W\subset \Delta^{(0)}$.

In particular, nontrivial homology classes in $H_{k-1}(\Delta_{W})$ are not only boundaries in $\Delta$, but one can find them to be boundaries of chains with at most $((k+1)g_{k+1}+(d+ 1-k)g_k)(\Delta)$ additional vertices.
\end{thm}

We again apply Alexander duality to obtain a dual result for $g_{k+1}(\Delta)\le 0$.

\begin{cor}
In the situation of Theorem~\ref{cor:betti}, we have $\widetilde{\beta}_{d-k+1}(\Delta_{W\setminus \mc{E}})=0$ for $W\subset \Delta^{(0)}$.
\end{cor}

\section{The generalized lower bound theorem}

We now give a new and simple proof of the generalized lower bound theorem, for simplicial polytopes and for Lefschetz spheres, which includes a new characterization in terms of toric chordality (similar to the one of Kalai for the special case $k=2$, cf.\ \cite{KalaiRig}).

\begin{thm}\label{thm:toricGLBCpolytopes}
Let $P$ denote any simplicial $d$-polytope. Then the following are equivalent:
\begin{compactenum}[\rm (A)]
\item $g_{k}(P)=0$ for some $2k \le d$.
\item $\partial P$ is toric $k$-chordal.
\item $P$ admits a $k$-stacked triangulation, i.e.\ a triangulation without interior faces of dimension $\le d-k$.
\end{compactenum}
\end{thm}

\begin{proof}
(c) $\Longrightarrow$ (a) is easy \cite{McMullenWalkup:GLBC-71}, (a) $\Longrightarrow$ (b) follows from Theorem~\ref{thm:stress_quant} (where $\mc{E}=\emptyset$). Finally, for (b) $\Longrightarrow$(c), we use the fact that $\Cl_k (P)$ is Cohen--Macaulay by Corollary~\ref{cor:propagation_SR}. The rest of the proof is as in \cite{MN}:
by the work of McMullen \cite{TSPMcMullen}, $\Cl_k (\partial P)$ is a geometric subcomplex of $\R^d$. By Corollary~\ref{cor:propagation_SR}, $\Cl_k (\partial P)$ is also Cohen--Macaulay, and thus triangulates~$P$.
\end{proof}

\begin{rem}\label{rem:glbt}
Virtually the same result holds for all proper geometric simplicial rational homology $(d-1)$-spheres $\Delta$ in $\R^d$ with the \Defn{weak Lefschetz property} (cf.~\cite{MMN}), i.e.\ for which there exists a linear differential $\upomega$ such that $\Sl_{k+1}(\Delta)\xrightarrow{\upomega} \Sl_k(\Delta)$ is, for every $k$, injective or surjective.

The only thing that needs amendment in Theorem~\ref{thm:toricGLBCpolytopes} is to weaken (c) to say that $\Cl_k (\Delta)$ is merely a $k$-stacked triangulated ball in homology. Only the implication (b) $\Rightarrow$ (c) needs elaboration, but it follows with Lemma~\ref{lem:torreschord}:

To see this, let $X= \Cl_k (\Delta)$ and $\Sigma= (v\ast \Delta) \cup X$, where $v$ is any new vertex (and naturally $v\ast \Delta$ and $X$ are identified along the common subcomplex $\Delta$), which we realize in $\R^{d+1}$ using generic coordinates. It follows with Lemma~\ref{lem:torreschord}, or even only isomorphism~\eqref{eq:iso_tay}, that $X$ is acyclic and $\Sigma$ is Cohen--Macaulay, so that the stress space in degree $d+1$ is generated by a single element $\mu_{\Sigma}$. Now, as a cone over a rational homology sphere, $\St_v\Sigma$ has a perfect Poincar\'e pairing with respect to its fundamental class $\mu_{\St_v\Sigma}$ in degree $d$ \cite{Grabe}, i.e.\ every stress of $\St_v \Sigma$ is a derivative of $\mu_{\St_v \Sigma}$. 

It follows with the cone lemma that every stress of $\St_v\Sigma$ is a derivative of $\mu_\Sigma$. But every stress of $X$ is supported in $\Delta$ by toric chordality and propagation, so that $\Sigma$ has perfect Poincar\'e pairing and is therefore a rational homology sphere. We conclude that $X$, as the deletion of $v$ from $\Sigma$, is a rational homology ball.
\end{rem}

\section{The balanced generalized lower bound theorem via partition of unity}

Let us turn our attention to balanced polytopes. A simplicial $(d-1)$-complex $\Delta$ is called \Defn{balanced} if there exists a ``coloring map''
\[\mathbf{c}:\Delta^{(0)}\ \longrightarrow\ [d]=\{1,\cdots,d\}\]
that is an injection on every facet of $\Delta$. If $S$ is a subset of $[d]$, we also use $\Delta_S:=\Delta_{\mathbf{c}^{-1} S}$ to denote the restriction of $\Delta$ to a specific color set.

The smallest balanced $(d-1)$-sphere is, quite clearly, the boundary of the crosspolytope. Klee and Novik \cite{KN} conjectured a far-reaching generalized lower bound conjecture for balanced polytopes and weak Lefschetz spheres.

The following result of Juhnke-Kubitzke and Murai \cite{JS} resolves the first part of this conjecture. The proof is based on a clever induction and heavy use of the weak Lefschetz property.

\begin{lem}[Juhnke-Kubitzke--Murai, cf.\ \cite{JS}]\label{lem:JS}
Let $\Delta$ denote the boundary of a balanced simplicial $d$-polytope, and let $S$ denote any subset of $[d]$. Let $\widetilde{\Delta}_S$ denote a generic projection of $\Delta_S$ to  an ${\# S}$-dimensional linear subspace. Then, for every $k\le \frac{\#S+1}{2}$, we have a surjection
\begin{equation} \label{eq:surjection_balanced}
\Sl_{k}(\widetilde{\Delta}_S)\ \xtwoheadrightarrow{\, \upomega\, }\ \Sl_{k-1}(\widetilde{\Delta}_S).
\end{equation}
for some generic $\upomega$.
\end{lem}

\begin{rem}
As we will see, we could choose $\upomega=\updelta$ in this lemma, though this is not necessary. 
\end{rem}

For purposes of self-containedness, we give a simple proof based on partition of unity. In its general form, it is stated as follows
If $\Delta$ is Cohen--Macaulay and proper, these facts are essentially equivalent by the following ``partitioning lemma'':
\begin{lem}[Partition of unity, \cite{AHK2}]\label{lem:partyyyyy}
	Consider a proper geometric Cohen--Macaulay $(d-1)$-complex $\Delta$ in $\R^d$. Then, for every $k<d$, we have a surjection
	\[\bigoplus_{v\in \Delta^{(0)}} \Sl_k(\St_v \Delta)\ \longtwoheadrightarrow\ \Sl_k(\Delta).\]
\end{lem}
\noindent For a lemma this simple to state, it is a little tricky to prove, but it will not be needed here in this strength: We will content ourselves with a simpler version. Let us start by giving an elementary version of Lemma~\ref{lem:partyyyyy}.

A simplicial complex is \Defn{pure} if all its facets are of the same dimension.
A pure $(d-1)$-dimensional simplicial complex $\Delta$ is \Defn{shellable} if there is a total order on the facets of $\Delta$ such that the intersection of every facet with the union of all subsequent facets is pure $(d-2)$-dimensional simplicial complex, or empty.

\begin{lem}\label{lem:partyyyyy2}
Consider a proper geometric shellable $(d-1)$-complex $\Delta$ in $\R^d$. Then, for every $k<d$, we have a surjection
\[\bigoplus_{v\in \Delta^{(0)}} \Sl_k(\St_v \Delta)\ \longtwoheadrightarrow\ \Sl_k(\Delta).\] 
\end{lem}

We refer to \cite{Z} for an introduction to shellability, and mention only that boundaries of simplicial polytopes are always shellable \cite{BruggesserMani} and that shellable complexes are always Cohen--Macaulay \cite{Stanley96}.

\begin{proof}
The proof is exceedingly simple: Consider a shelling step $\Delta$ to $\Delta'$, consisting of the removal of a single facet $\sigma$ from $\Delta$. Let $\tau$ denote the minimal face of $\Delta$ not in $\Delta'$. 

It is well-known and easy to see that \[\Sl_\ast(\Delta')\ \longrightarrow\ \Sl_\ast(\Delta)\] is an isomorphism in every degree except for $k:=\# \tau$, where it is a injection with a one-dimensional cokernel. Any stress in $\Sl_\ast(\Delta)$ not in $\Sl_\ast(\Delta')$ therefore is a $k$-stress that contains $\tau$ in its support.

Let $v$ denote any vertex of $\Delta$ not in $\sigma$, which exists unless $k=d$. The same then applies for \[\Sl_\ast(\St_v \Delta')\ \longrightarrow\ \Sl_\ast(\St_v\Delta).\] Choose a stress $\gamma\in \Sl_k(\St_v \Delta)$ representing the cokernel of this map. Then every stress in $\Sl_\ast(\Delta)$ is supported in $\Sl_\ast(\Delta')$ modulo a scalar multiple of $\gamma$, so that the theorem follows by induction.
\end{proof}

To prove Lemma~\ref{lem:JS}, we use the same idea:

\begin{lem}\label{lem:partey}
Consider a shellable balanced $(d-1)$-complex $\Delta$ with a proper embedding into $\R^{d-1}$. Then we have a surjection
\[\bigoplus_{v\in \Delta_{\{d\}}} \Sl(\Lk_v \Delta)\ \longtwoheadrightarrow\ \Sl(\Delta_{[d-1]}).\]
\end{lem}

\begin{proof}
The proof is the same as in the case of Lemma~\ref{lem:partyyyyy}; we just have to observe that the shelling of $\Delta$ restricts to a shelling of $\Delta_{[d-1]}$ and observe the changes in stresses in every step.
\end{proof}

\begin{rem}
This extends to Cohen--Macaulay complexes, see \cite{AHK2}.
\end{rem}

It follows that every stress can be partitioned into stresses supported in links of vertices of any chosen color.

\begin{proof}[\textbf{Proof of Lemma~\ref{lem:JS}}]
If we assume now that $\Delta$ is the boundary of a simplicial balanced polytope, then we in turn apply the hard Lefschetz theorem to every link. 

The claim of Lemma~\ref{lem:JS}, and thereby the first part of the Klee--Novik balanced generalized lower bound conjecture, follows immediately for $|S|=d-1$. Iterating the argument proves Lemma~\ref{lem:JS} in full generality.
\end{proof}

\section{The case of equality}

The second part of the Klee--Novik conjecture concerned a conjecture for the case of equality, which asserts the existence of a certain nice triangulation for $\Delta$, that can be constructed in three steps: Consider the boundary $\Delta$ of a balanced $d$-polytope  on color set $[d]$, and let $k$ denote a non-negative integer. The construction consists of three steps

\begin{compactenum}[\rm (A)]
\item Consider any pair of vertices $(v, v')$ in $\Delta$ of the same color, called \Defn{antipodes}. Let
\[\mathcal{I}_{v,v'}\ :=\ \{v,v'\}\ast (\St_v \Delta\cap \St_{v'} \Delta)^{(\ge k-2)},\] where $C^{(\ge j)}$ denotes the subcomplex of $C$ induced by faces of dimension $\ge j$, and $\{v,v'\}\ast X$ denotes the suspension with apices $v, v'$. 
If the star $\St_v \mathcal{I}_{v,v'}$ has no $(k-1)$-stress, proceed no further. Otherwise, let $\mbf{o}_{v,v'}$
denote a vertex created anew (and given color $(d+1)$), the \Defn{center of the antipode}, and let
\[\mathcal{A}_{v,v'}\ :=\ \mbf{o}_{v,v'} \ast \mathcal{I}_{v,v'}.\]
We attach $\mathcal{A}_{v,v'}$ to $\Delta$ using the natural embedding $\mathcal{I}_{v,v'}\hookrightarrow \Delta$. This is repeated with every antipode in $\Delta$.


\item Consider now $(v,v')$, $(w,w')$ two antipodes of $\Delta$ in the resulting complex. If ${w,w'} \subset \mathcal{I}_{v,v'}$, or (equivalently) ${v,v'} \subset \mathcal{I}_{w,w'}$, then we now identify $\mbf{o}_{v,v'}=\mbf{o}_{w,w'}$. Repeat this with every pair of antipodes.
\item Take the $k$-clique complex of the resulting simplicial complex. We obtain a simplicial complex $\mr{Cl}^{\mathbf{b}}_{k} {\Delta}$, the \Defn{balanced $k$-clique complex}.
\end{compactenum}

Note that for every center of the antipode, we have a natural $\mathbb{Z}_2$ action permuting the elements of the antipode.

Observe that $\mr{Cl}^{\mathbf{b}}_{k} {\Delta}$ is a canonical balanced triangulation of the cell complex defined by Klee--Novik \cite{KN} in proposal of their balanced lower bound conjecture. Whenever this construction is used in the context of a geometric simplicial complex, the centers of antipodes are put at the origin.
For a color $c$, we use $\updelta_c$ to denote the sum of partial differentials over vertices with color $c$.

\begin{lem}\label{lem:balanced_prop}
Let $\Delta$ be the boundary of a balanced simplicial $d$-polytope, and assume that, for some $k\le \left\lfloor\frac{d}{2}\right\rfloor$, and for all $S\subset [d]$, $\#S=2k-1$, the surjection of Equation~\eqref{eq:surjection_balanced} is also an injection.

Then 
\begin{compactenum}[(1)]
\item the map is also an injection for all $\ell \ge k$, $\ell\le \left\lceil\frac{d}{2}\right\rceil$, 
\item and $\mc{T}:=\mr{Cl}^{\mathbf{b}}_{k} {\Delta}$ gives a surjection
\[ \Sl_{\ell}( \mc{T}) \xrightarrow{\ \updelta_{d+1}\ }   \Sl_{\ell-1}(\mc{T})\ \longrightarrow\ 0\]
for all $\ell$. 
\item Moreover, if $\ell \ge k$, we have an exact sequence \[0\ \longrightarrow\ \Sl_{\ell}(\mc{T}_{[d]})\ \longrightarrow\ \Sl_{\ell}( \mc{T}) \xrightarrow{\ \updelta_{d+1}\ }   \Sl_{\ell-1}(\mc{T})\ \longrightarrow\ 0.\]
\end{compactenum}
\end{lem}


The polytopality assumption can be weakened: Similar to Lemma~\ref{lem:JS}, the lemma also holds in the generality of weak Lefschetz spheres; one reason we do not adopt this level of generality here is that we want to remain as self-contained as possible, and rely on shellings heavily. That said, everything needed is a generalization of Lemma~\ref{lem:partey} for balanced Cohen--Macaulay complexes (see \cite{AHK2}).

Lemma~\ref{lem:balanced_prop} is essentially the balanced analogue of Theorem~\ref{thm:toric_propagation}; we do only not call it such to express a mild dissatisfaction with it. Indeed, as we shall see in the proof, it is critically used that our complex $\Delta$ is a shellable sphere (or at least $2$-Cohen--Macaulay) already, which is not needed in the original theorem. Indeed, only the final part of this proof could be justifiably called a balanced propagation theorem, albeit requiring somewhat special assumptions.

\begin{proof}[\textbf{Proof of Lemma~\ref{lem:balanced_prop}}]
\textbf{To start with}, the claim (1) is proven by Klee--Novik \cite{KN} (and it also follows easily using our proof of Lemma~\ref{lem:JS}). 


\medskip

\textbf{Secondly}, note that $\Delta$ has no missing $i$-faces, $i=k,\cdots,d-k$ (that is has different colors under $\mathbf{c}$ if $i=k=1$). 
This follows as in Lemma~\ref{lem:torreschord}, but let us give a more direct argument: 

Let us first assume that $i\le \lceil\frac{d}{2}\rceil$.
Assume there is a missing $i$-face, let $z$ be the $(i-1)$-cycle defined by its boundary. Consider furthermore $S$ a colorset of cardinality $2i-1$ that contains all but one of the colors of $z$. Consider finally a generic projection $\wt{\Delta}$ of $\Delta$ to $\R^{2i-1}$ and $v$ a vertex of color $c\notin S$ of $z$. Let ${S'}=S\cup \{c\}$.

We have a natural embedding
${\Lk_v \Delta}_{S'}\hookrightarrow {\Delta}_S$, inducing an embedding of stress spaces
\[\Sl(\Lk_v \wt{\Delta}_{S'})\ \longhookrightarrow\ \Sl(\wt{\Delta}_S).\]
We therefore obtain a commutative square
\begin{equation}\label{eq:iso_B}
\begin{tikzcd}
 \Sl_{i}(\wt{\Delta}_S) \arrow[hook, two heads]{r}{\ \upomega\ } &  \Sl_{i-1}(\wt{\Delta}_S) \\
 \Sl_i(\Lk_v \wt{\Delta}_{S'}) \arrow[hook, two heads]{r}{\ \upomega\ } \arrow[hookrightarrow]{u} & \Sl_{i-1}(\Lk_v \wt{\Delta}_{S'}) \arrow[hookrightarrow]{u}
 \end{tikzcd}
\end{equation}
where the horizontal maps are isomorphisms. Since, in a generic embedding, every face and every complete cycle is contained in the support of some nontrivial stress, it follows that $\Lk_v z$ is the boundary of a $(i-1)$-face in $\Lk_v \Delta$ if and only if it bounds such a face in ${\Delta}_S$. 

If $i> \lceil\frac{d}{2}\rceil$, $i\le d-k$, consider $j=2i-1-d$, and consider a face $\sigma$ of cardinality $j$ in $\Delta$. Then, following what we proved above, $\Lk_\sigma (\Delta)$ has no missing faces of dimension $i-j$, so that $\Delta$ has no missing face of dimension $i$ containing $\sigma$.

This proves the claim.

\medskip

As a \textbf{third step}, we progress to prove that the map
\[\updelta_{d+1}: \Sl_{\ell}(\mc{T})\longrightarrow \Sl_{\ell-1}({\mc{T}})\] is surjective; we will prove this first for $\ell\le k+1$, and then use a balanced version of toric propagation to extend this to all $\ell$. 

Consider a colorset $R$ of cardinality $2k$ in $[d]$. 

We need a lemma of the partitioning type. For simplicity, we call a pair $(X,\upomega)$ \Defn{$i$-rigid} if
\[\Sl_{i} (X) \xrightarrow{\upomega} \Sl_{i-1} (X) \]
is a surjection. We say it is \Defn{minimally $i$-rigid} if the map is an isomorphism. We say that it is \Defn{trivially $i$-rigid} if $\Sl_{i-1} (X)$ is trivial.

\begin{lem}\label{lem:partitioning}
Let $v$ denote any vertex of $\Delta_R$ of color $c\in R$, let $S=R\setminus \{c\}$, and $\wt{\Delta}_R$ denote a generic projection of $\Delta_R$ to $\R^{2k-1}$. Then
\begin{equation}\label{eq:lrm}
\bigoplus_{\substack{w\in \wt{\Delta}_{\{c\}}\\ v \neq w }} \Sl_{i} (\Lk_v \wt{\Delta}_R \cap \Lk_w \wt{\Delta}_R)\ \longrightarrow \Sl_{i} (\Lk_v \wt{\Delta}_R) 
\end{equation}
is an isomorphism for $i=k,k-1$.
\end{lem}

\begin{proof}
Observe first that for every collection $\mathcal{C}$ of color $c$, we have an isomorphism
\begin{equation}\label{eq:lp}\upomega:\Sl_{k} (\Lk_{\mathcal{C}} \wt{\Delta}_R)\ \longisorightarrow\ \Sl_{k-1} (\Lk_{\mathcal{C}} \wt{\Delta}_R),
\end{equation}                         
where $\Lk_{\mathcal{C}} \wt{\Delta}_R:=\bigcap_{u\in \mathcal{C}} \Lk_{u} \wt{\Delta}_R$.

Indeed; this map is clearly an injection by assumption. Moreover, for $\mathcal{C}$ of cardinality at most one, we obtain surjectivity by Lemma $\ref{lem:JS}$. Consider now some multiindex $\mathcal{C}=\{i_1,\cdots,i_n\}$ and a stress $\gamma$ in $\Lk_{\mathcal{C}} \wt{\Delta}_R$. By Lemma~\ref{lem:JS}, we have that $\gamma$ has a preimage in $\Lk_{i_j} \wt{\Delta}_R$, $j=1,\cdots, n$. Since $\upomega$ is injective, the preimages must coincide, and we conclude that the preimage is supported in $\Lk_{\mathcal{C}} \wt{\Delta}_R$. 

We also conclude that both conclusions of Lemma~\ref{lem:partitioning} are equivalent, so that we now restrict to proving the case $i=k-1$.

\textbf{Surjectivity:} We follow  Kalai's argument for the chordality of polytopes with $g_2=0$, and more detailedly the generalization by Adiprasito--Nevo--Samper \cite[Lemma A.1]{ANSII}. Abbreviate $\wt{\Delta}_S:=(\wt{\Delta}_R)_{S}$, i.e.\ the restriction of $\wt{\Delta}_R$ to colorset $S$. 

Consider the cover of $\wt{\Delta}_S$ by $\Lk_v \wt{\Delta}_R$, where $v$ stands for vertices of $\Delta$ of color $c$. In particular, the generalized Mayer-Vietoris principle induces a resolution of the Ishida complex for the stress space of $\wt{\Delta}_S$ given by its coordinates and $\upomega$. Now, to understand the total complex of the resulting double complex, we can make two simple observations
\begin{compactitem}[$\circ$]
\item Computing the homology of the total complex can be done in an easy way for instance by using first the fact that the resolution induced by the cover is exact, so that we are only left with the Ishida complex for $\wt{\Delta}_S$  given by its coordinates and $\upomega$ easily; its top homology group (i.e.\ in degree $k$) is therefore isomorphic to
\[\ker\, [\upomega:\Sl_{k}(\wt{\Delta}_S)\ \rightarrow\ \Sl_{k-1}(\wt{\Delta}_S)].\]
\item On the other hand, the degree $(k-1)$-column of the double complex contains as a direct summand the resolution of the Ishida complex for $\wt{\Delta}_S$ with respect to its coordinates only:
\begin{equation}\label{eq:irc}
\cdots \, \rightarrow\, \bigoplus_{\substack{w,v\in \wt{\Delta}_{\{c\}}\\ v \neq w }} \Sl_{i} \Sl_{k-1}(\Lk_v \wt{\Delta}_R \cap \Lk_w \wt{\Delta}_R)\, \rightarrow\, \bigoplus_{v\in \wt{\Delta}_{\{c\}}} \Sl_{k-1} (\Lk_v \wt{\Delta}_R)\, \rightarrow\, \Sl_{k-1} (\Lk_v \wt{\Delta}_S)\, \rightarrow\, 0
\end{equation}
whose entries vanish under the differential of the Ishida complex; 
the first homology of this complex is therefore injects into the top homology of the total complex.
\end{compactitem}

To see how this implies surjectivity, note that following Brugesser--Mani \cite{BruggesserMani}, $\wt{\Delta}_R$ can be shelled in such a way that removes the star of $v$ first. It follows with Lemma~\ref{lem:partey} that we have a surjection
\begin{equation*}\label{eq:cm}
\bigoplus_{\substack{w\in \wt{\Delta}_{\{c\}}\\ w \neq v }} \Sl_{k-1} (\Lk_w \wt{\Delta}_R)\ \longtwoheadrightarrow\ \Sl_{k-1} ((\wt{\Delta}_R)_{R\setminus\{c\}}) 
\end{equation*}
so that we write $s$ as a sum of $(k-1)$-stresses supported in $\Lk_w \wt{\Delta}_R$ vertices $w \neq v$ of color~$c$. Now, $\upomega: \Sl_k (\Lk_w \wt{\Delta}_R)\rightarrow \Sl_{k-1} (\Lk_w \wt{\Delta}_R)$ is surjective. At the same time $\upomega: \Sl_k (\Lk_v \wt{\Delta}_R)\rightarrow \Sl_{k-1} (\Lk_v \wt{\Delta}_R)$ is surjective.

\begin{figure}[h!tb]
	\begin{center}
		\includegraphics[scale = 0.6]{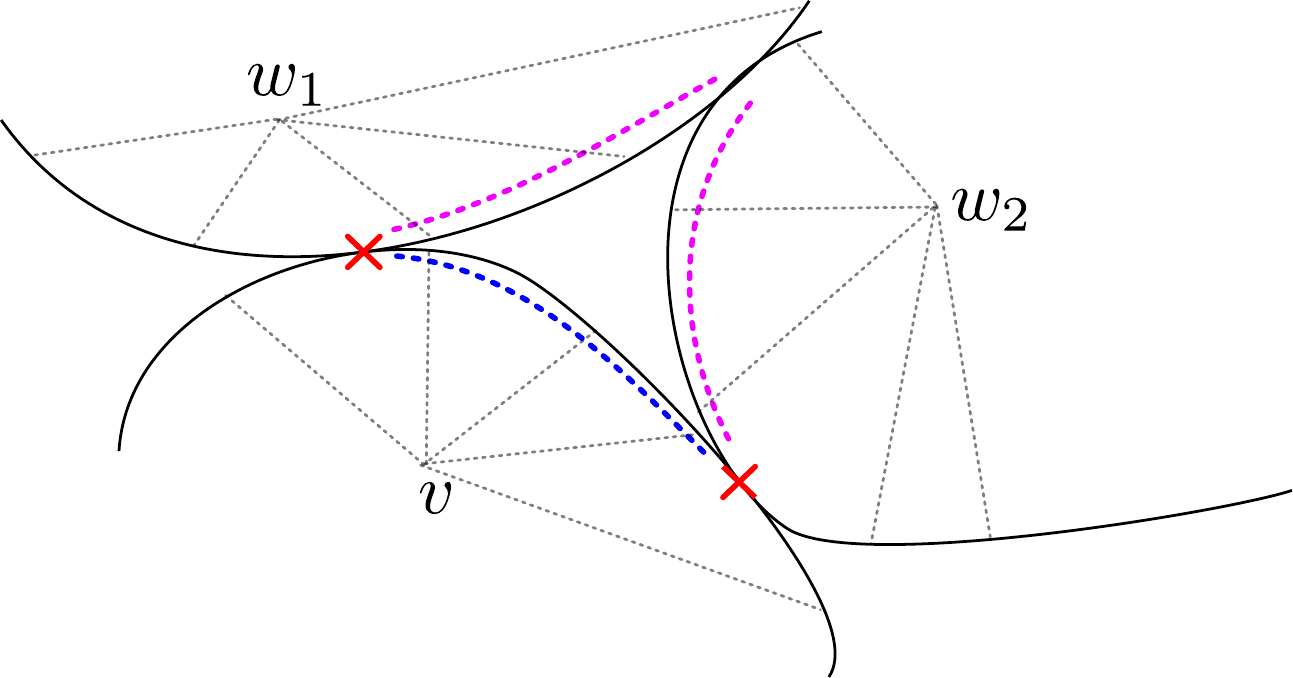}
		\caption{A non-partitionable $0$-cycle, boundary of two $1$-chains.}
		\label{fig:surj}
	\end{center}
\end{figure}

Hence, $s$ is in the image of two different $k$-stresses unless the decomposition of $s$ above is already supported in the link of $v$, see Figure~\ref{fig:surj} unless it defines a homology class in $H_1$ of the chain complex \eqref{eq:irc}. The former, however, is excluded by the injectivity of $\upomega: \Sl_k (\wt{\Delta}_R)\rightarrow \Sl_{k-1} (\wt{\Delta}_R)$.

\begin{rem}
In fact, the reasoning above proves that the chain complex \eqref{eq:irc} is exact, which can be used to simplify the next step. We shall not need this.
\end{rem}

\textbf{Injectivity:}
The kernel of the partitioning map~\eqref{eq:lrm} is spanned by $(k-1)$-stresses in \[\Lk_{v,w,\overline{w}} \wt{\Delta}_R\ :=\ \Lk_{v} \wt{\Delta}_R\cap \Lk_{w} \wt{\Delta}_R\cap \bigcup_{\substack{u \in \wt{\Delta}_{\{c\}} \\ u\neq v,w }} \Lk_{u}\wt{\Delta}_R\] for $v\neq w$ of color $c$; hence, we simply have to show that $\Lk_{v,w,\overline{w}} \wt{\Delta}_R$ is trivially $k$-rigid. But we already observed that the complexes $\Lk_{v,w} \wt{\Delta}_R=\Lk_{v} \wt{\Delta}_R\cap \Lk_{w} \wt{\Delta}_R$ are $k$-rigid for every choice of $v\neq w$ of color $c$ (this is isomorphism~\eqref{eq:lp} for two-element sets). Moreover $\Lk_{v} \wt{\Delta}_R$ is minimally $k$-rigid (this is isomorphism~\eqref{eq:lp} for one-element sets). In particular, \[\Lk_{v} \wt{\Delta}_R\cap\Lk_{w} \wt{\Delta}_R\  \text{and}\ \Lk_{v} \wt{\Delta}_R\cap\bigcup_{\substack{u \in \wt{\Delta}_{\{c\}} \\ u\neq v,w }} \Lk_{u}\wt{\Delta}_R\] are $k$-rigid. But then their intersection $\Lk_{v,w,\overline{w}} \wt{\Delta}_R$ must be trivially $k$-rigid by the Mayer--Vietoris formula for the Ishida complex.
\end{proof}

It is useful to summarize this lemma in terms of the differential $\updelta_{d+1}$.

\begin{cor}\label{cor:part}
We have a surjection
\begin{equation}\label{eq:bt}
\updelta_{d+1}: \Sl_{i+1}(\mc{T})\ \longtwoheadrightarrow\  \Sl_{i}({\mc{T}}).
\end{equation}
for $i\le k$.
Moreover, we have  
\begin{equation}\label{eq:btl}
\ker[\updelta_{d+1}: \Sl_{i+1}(\mc{T})\ \longtwoheadrightarrow\  \Sl_{i}({\mc{T}})]\ = \ \Sl_{i+1}(\mc{T}_{[d]}).
\end{equation}
for $i=k,\ k-1$.
\end{cor}

We are therefore ``almost'' at toric $k$-chordality, as we do not have to injectivity of the map~\eqref{eq:btl}, but a controlled kernel.

\begin{proof}
We start by proving the first claim for $i\le k-1$.
For every colorset $R\subset [d]$ of cardinality $2k$, $R'=R\cup\{d+1\}$, $c\in R$, $S'=R'\setminus\{c\}$ and a generic projection $\widetilde{\mc{T}}$ of $\mc{T}$ to $\mathbb{R}^{2k-1}$, the previous lemma gives a surjection
\[\bigoplus_{\substack{v,w\in \wt{\Delta}_{\{c\}}\\ v \neq w }} \Sl_{k-1} (\Lk_v \wt{\mc{T}}_R \cap \Lk_w \wt{\mc{T}}_R)\ \longtwoheadrightarrow \Sl_{k-1} ( \wt{\mc{T}}_S)\]
for every vertex $v$ of $\Delta_R$ of color $c\in R$. By surjectivity of $\upomega:\Sl_\ell (\Delta) \rightarrow \Sl_{\ell-1} (\Delta)$  for $\ell \le \frac{d}{2}$, (or even just perfectness of the Poincar\'e pairing in $\Delta$), we have in particular also
\[\bigoplus_{\substack{v,w\in \wt{\Delta}_{\{c\}}\\ v \neq w }} \Sl_{i} (\Lk_v \wt{\mc{T}}_R \cap \Lk_w \wt{\mc{T}}_R)\ \longtwoheadrightarrow \Sl_{i} ( \wt{\mc{T}}_S)\]
for all $i\le k-1$.
Since the complexes $\{v,w\} \ast (\Lk_v \wt{\Delta} \cap \Lk_w \wt{\Delta})$ and $ \mathcal{I}_{v,w}$ coincide up to their $(k-2)$-skeleton, we obtain a surjection 
\[\updelta_{d+1}: \Sl_{i+1}(\widetilde{\mc{T}}_{S'})\ \longtwoheadrightarrow\  \Sl_{i}(\widetilde{\mc{T}}_{S'})\]
for all $i\le k-1$.
 By partition of unity, we have a surjection
\[\bigoplus_{\sigma \in \Delta^{(d-2k)}} \Sl_{k-1}(\St_\sigma \Delta)\ \longtwoheadrightarrow\ \Sl_{k-1}(\Delta)\]
so that it follows then that we have a surjection
\begin{equation*}
\updelta_{d+1}: \Sl_{k}(\mc{T})\ \longtwoheadrightarrow\  \Sl_{k-1}({\mc{T}})
\end{equation*}
so the first claim of the lemma follows. 

For the characterization of the kernel, consider $\wt{\mc{T}}$ denote a generic projection of ${\mc{T}}$ to $\R^{2k-1}$, so that the map 
\[
\updelta_{d+1}: \Sl_{k}(\wt{\mc{T}}_{(R\cup \{d+1\})\setminus\{c\}})\ \longtwoheadrightarrow\  \Sl_{k-1}(\wt{\mc{T}}_{(R\cup \{d+1\})\setminus\{c\}}).
\]
has kernel $\Sl_{k}(\wt{\mc{T}}_{R\setminus\{c\}})$, so Corollary~\ref{cor:part} for $i=k-1$ follows.

The claim for $i=k$ follows directly by isomorphism~\eqref{eq:lp}.
\end{proof}

%

Hence we are only left to examine the behavior of $\updelta_{d+1}$ for indices $i\ge k$. It is at this point that we forget about the original Lefschetz element $\upomega$ entirely, and focus on $\updelta_{d+1}$.

\medskip

\textbf{Finally}, we get to balanced toric propagation. It follows at once as in the proof of propagation of weak toric chordality (Corollary~\ref{cor:quant2}) that~\eqref{eq:btl} holds for $i\ge k-1$. We wish to prove that map~\eqref{eq:bt} is
a surjection for all $i\ge k-1$ as well; assume it is proven for $i= \ell, \ell-1$. 

Consider an $(\ell+1)$-stress $\gamma$, $\ell \ge \frac{d}{2}$, of $\mc{T}$, and its restriction $\gamma_v$ to $\St_v^\circ \mc{T}$, where $v$ is any vertex of color in $[d]$. 

By induction assumption $\updelta_v \gamma=\frac{\mr{d}}{\mr{d} x_v} \gamma$ is in the image of $\updelta_{d+1}$. Let $\alpha_v$ denote a preimage of $\updelta_v \gamma$ under $\updelta_{d+1}$.

To show that $\gamma_v$ is in the image of $\Sl_{\ell+2}(\St_v^\circ \mc{T})$ under $\updelta_{d+1}$, we have to show that $\alpha_v$ is, modulo a $(\ell+1)$-stress in $\mc{T}_{[d]}$, in the image of $\updelta_v$.

It is sufficient to show that $\updelta_w \alpha_v\in \Sl_{\ell}(\St_w \mc{T})$ is in the image of $\updelta_v$ for every $w$ of color in $[d]$. 
To this end, we consider in analogy to the toric propagation the commutative diagram
\[
\begin{tikzcd}
& \Sl_{\ell}( \mc{T}) \arrow{r}{\ \updelta_{d+1}\ } &   \Sl_{\ell-1}(\mc{T})\\
\Sl_{\ell+1}(\St_v^\circ \mc{T}) \arrow{r}{\ \updelta_{d+1}\ } \arrow{ur}{\updelta_v} & \Sl_{\ell}(\St_v^\circ \mc{T}) \arrow{ur}{\updelta_v} &
\end{tikzcd}
\]
where the top map has kernel $\Sl_{\ell}(\mc{T}_{[d]})$.
Hence, $(\updelta_{d+1}^{-1}\circ\updelta_v^{-1}\circ\updelta_{d+1}) (\updelta_w \alpha_v)$ gives the desired preimage of $\updelta_w \alpha_v$ as in the proof of the classical propagation theorem. We therefore obtained a preimage of $\gamma_v$ under $\updelta_{d+1}$ in $\Sl_{\ell+2}(\St_v^\circ \mc{T})$.

To conclude the surjection \[\updelta_{d+1}: \Sl_{\ell+2}(\mc{T})\ \longtwoheadrightarrow{}\  \Sl_{\ell+1}({\mc{T}})\] we simply observe again that the above construction gives local charts for the preimage of $\gamma$ which canonically glue together.
\end{proof}

We conclude with the second part of the Klee--Novik conjecture \cite{KN}.

\begin{thm}
In the situation of Lemma~\ref{lem:balanced_prop}, 
\begin{compactenum}[\rm (A)]
\item $\mc{T}$ is a rational homology ball. Moreover,
\item the constructed triangulation is balanced $k$-stacked, i.e., all faces of dimension $\le d-k$ on color set $[d]$ in  $ {\mc{T}}$ are boundary faces, and
\item the link of every vertex of color $d+1$ is, combinatorially, the boundary of a crosspolytope.
\end{compactenum}
\end{thm}

\begin{proof}
For the proof of claim (A), we argue as in Remark~\ref{rem:glbt}: As in Corollary~\ref{cor:propagation_SR}, it follows that $\mc{T}$ is Cohen--Macaulay.
	
Since the kernel of the map~\eqref{eq:bt} is $0$-dimensional for $i=d+1$ because \[\ker[\updelta_{d+1}: \Sl_{d+1}(\mc{T})\ \longtwoheadrightarrow\  \Sl_{d}({\mc{T}})]\ = \ \Sl_{d+1}(\mc{T}_{[d]})\ =\ 0.
\] we observe with isomorphism~\eqref{eq:iso_tay} that $\mc{T}$ is acyclic. The same holds for links of faces on colorset $[d]$ up to dimension $d-k$ by the Cone Lemma and as in Lemma~\ref{lem:quant2} and Theorem~\ref{thm:stress_quant}. 

Let now $\Sigma=(v\ast \Delta)\cup \mc{T}$, which we realize with generic coordinates in $\R^{d+1}$ using a new vertex $v$. By the observation we just made and isomorphism~\eqref{eq:iso_tay}, for every face $\sigma$ of $\mc{T}_{[d]}$ of dimension at most $d-k$, $\Sl_{d-\dim \sigma}(\St_\sigma(X))$ is generated by a single element $\mu_\sigma$. By partition of unity (Lemma~\ref{lem:partey}), we have
\[\bigoplus_{\sigma\in \Sigma^{(\ell)}}  \Sl_{d-\ell}(\St_\sigma(\Sigma))\ \longtwoheadrightarrow\ \Sl_{d-\ell}(\Sigma)\] and
we conclude that every stress of degree at least $d-(d-k)=k\le \left\lfloor\frac{d}{2}\right\rfloor$ is a derivative of $\mu_{\{\varnothing\}}$.

Moreover, since $\Delta^{(\le d-k)}=\mc{T}_{[d]}^{(\le d-k)}$, we have a surjection 
\[\Sl_{\ell} (\St_v \Sigma)\ \longtwoheadrightarrow\ \Sl_{\ell} (\Sigma)\]
for all $\ell\le d-k+1$. But $\Delta$ is a rational homology sphere \cite{Grabe}, so every stress of degree at most $d-k+1\ge \left\lceil\frac{d}{2}\right\rceil$ in $\Sigma$ is a derivative of $\mu_{\{\varnothing\}}$. Hence we have a perfect Poincar\'e pairing in $\Sigma$, so that $\Sigma$ is a (rational) homology sphere. Removing $v$ yields a (rational) homology ball $\mc{T}$, the first claim follows.

Claim (B) follows by construction of $\mc{T}$.

Finally, the link of every center admits a fix-point free $\mathbb{Z}_2$ that is induced as the diagonal embedding of the $\mathbb{Z}_2$-action on antipodes of a given color; as every individual action was transitive on the colorset, so is the diagonal embedding. Hence the link of every center of an antipode must be a crosspolytope.
\end{proof}

{
\providecommand{\bysame}{\leavevmode\hbox to3em{\hrulefill}\thinspace}
\providecommand{\MR}{\relax\ifhmode\unskip\space\fi MR }
\providecommand{\MRhref}[2]{%
	\href{http://www.ams.org/mathscinet-getitem?mr=#1}{#2}
}
\providecommand{\href}[2]{#2}

}

\end{document}